\documentclass[11pt]{amsart}

\usepackage{amssymb}

\newtheorem{thm}{Theorem}

\newtheorem{prop}{Proposition}

\newtheorem{note}{Remark}
\newtheorem{corr}{Corollary}

\begin{document}

\title[Comparison of topologies \ldots]{Comparison of topologies on
  $*$-algebras of locally measurable operators}

\author{V. I. Chilin}

\author{M.~A.~Muratov}



\date{}



\begin{abstract}
  We consider the locally measure topology $t(\mathcal{M})$ on the $*$-algebra $LS(\mathcal{M})$ of
  all locally measurable operators affiliated with a von Neumann
  algebra $\mathcal{M}$.  We prove that $t(\mathcal{M})$ coincides
  with the $(o)$-topology on $LS_h(\mathcal{M})=\{T\in
  LS(\mathcal{M}): T^*=T\}$ if and only if the algebra $\mathcal{M}$
  is $\sigma$-finite and a finite  algebra. We study
  relationships between the topology $t(\mathcal{M})$ and various
  topologies generated by faithful normal semifinite traces on
  $\mathcal{M}$.
\end{abstract}

\maketitle

\section*{Introduction}

The development of integration theory for a faithful normal
semifinite trace $\tau$ defined on a von Neumann algebra
$\mathcal{M}$ has led to a need to consider the $*$-algebra
$S(\mathcal{M},\tau)$ of all $\tau$-measurable operators affiliated
with $\mathcal{M}$, see, e.g.,~\cite{Nels.}. This algebra is a
solid $*$-subalgebra of the $*$-algebra $S(\mathcal{M})$ of all
measurable operators affiliated with $\mathcal{M}$. The $*$-algebra
$S(\mathcal{M})$ was introduced by I.~Segal~\cite{Seg.} to describe
a ``noncom\-mutative version'' of the \hbox{$*$-algebra} of
measurable complex-valued func\-tions. If $\mathcal{M}$ is a
commutative von Neumann algebra, then $\mathcal{M}$ can be
identified with the $*$-algebra $L_\infty(\Omega,\Sigma,\mu)$ of all
essentially bounded measurable complex-valued functions defined on a
measure space $(\Omega,\Sigma,\mu)$ with a measure $\mu$ having the
direct sum property. In this case, the $*$-algebra $S(\mathcal{M})$
is identified with the $*$-algebra $L_0(\Omega,\Sigma,\mu)$ of all
measurable complex-valued functions defined on
$(\Omega,\Sigma,\mu)$~\cite{Seg.}.

The $*$-algebras $S(\mathcal{M},\tau)$ and $S(\mathcal{M})$ are
substantive examples of $EW^*$-algebras $E$ of closed linear
operators, affiliated with the von Neumann algebra $\mathcal{M}$,
which act on the same Hilbert space $\mathcal{H}$ as $\mathcal{M}$
and have the bounded part $E_b=E\cap \mathcal{B}(\mathcal{H})$
coinciding with $\mathcal{M}$~\cite{Dixon}, where
$\mathcal{B}(\mathcal{H})$ is the $*$-algebra of all bounded linear
operators on $\mathcal{H}$. A natural desire of obtaining a maximal
$EW^*$-algebra $E$ with $E_b=\mathcal{M}$ has led to a construction
of the $*$-algebra $LS(\mathcal{M})$ of all locally measurable
operators affiliated with the von Neumann algebra $\mathcal{M}$,
see, for example,~\cite{Yead1}. It was shown in~\cite{Chil.Z.2} that
any $EW^*$-algebra $E$ satisfying $E_b=\mathcal{M}$ is a solid
$*$-subalgebra of $LS(\mathcal{M})$.

In the case where there exists a faithful normal finite trace $\tau$
on $\mathcal{M}$, all three $*$-algebras $LS(\mathcal{M})$,
$S(\mathcal{M})$, and $S(\mathcal{M},\tau)$
coincide~\cite[\S\,2.6]{Mur_m}, and a natural topology that endows
these $*$-algebras with the structure of a topological $*$-algebra
is the measure topology induced by the trace
$\tau$~\cite{Nels.}.  If $\tau$ is a semifinite but not a finite
trace, then one can consider the $\tau$-locally measure topology $t_{\tau \,l}$ and the weak
$\tau$-locally measure $t_{w\tau\,l}$ ~\cite{Bik2}. However, in the case where
$\mathcal{M}$ is not of finite type, the multiplication is not
jointly continuous in the two variables with respect to these
topologies. In this connection, it makes sense to use, for the
$*$-algebra $LS(\mathcal{M})$, the locally measure topology $t(\mathcal{M})$,
which was defined in~\cite{Yead1} for
any von Neumann algebras and which endows $LS(\mathcal{M})$ with the
structure of a complete topological
$*$-algebra~\cite[\S\,3.5]{Mur_m}.

The natural partial order on the selfadjoint part
$LS_h(\mathcal{M})=\{T\in LS(\mathcal{M}): T^*=T\}$ permits to
define, on $LS_h(\mathcal{M})$, an order convergence,
$(o)$-convergence, and the generated by it $(o)$-topology
$t_o(\mathcal{M})$. If $\mathcal{M}$ is a commutative von Neumann
algebra, $t(\mathcal{M})\leqslant t_o(\mathcal{M})$ and
$t(\mathcal{M})=t_o(\mathcal{M})$ on $LS_h(\mathcal{M})$ if and only
if $\mathcal{M}$ is of $\sigma$-finite algebra ~\cite[Ch.\,V, \S\,6]{Sar.}.
For noncommutative von Neumann algebras, such relations between the
topologies $t(\mathcal{M})$ and $t_o(\mathcal{M})$ do not hold in
general. For example, if $\mathcal{M}=\mathcal{B}(\mathcal{H})$,
then $LS(\mathcal{M})=\mathcal{M}$ and the topology $t(\mathcal{M})$
coincides with the uniform topology that is strictly stronger than
the $(o)$-topology on $\mathcal{B}_h(\mathcal{H})$ if
$\dim(\mathcal{H})=\infty$~\cite[\S\,3.5]{Mur_m}.

In this paper, we study relations between the topology
$t(\mathcal{M})$ and the topologies $t_{\tau\,l}$, $t_{w\tau\,l}$,
and $t_o(\mathcal{M})$. We find that the topologies $t(\mathcal{M})$
and $t_{\tau\,l}$ (resp. $t(\mathcal{M})$ and $t_{w\tau\,l}$)
coincide on $S(\mathcal{M},\tau)$ if and only if $\mathcal{M}$ is finite,
and $t(\mathcal{M})=t_o(\mathcal{M})$ on
$LS_h(\mathcal{M})$ holds if and only if $\mathcal{M}$ is a
$\sigma$-finite and finite. Moreover, it turns out that
the topology $t_{\tau l}$ (resp. $t_{w \tau\,l}$) coincides with the
$(o)$-topology on $S_h(\mathcal{M}, \tau)$ only for finite traces.
We give necessary and sufficient conditions for the topology
$t(\mathcal{M})$ to be locally convex (resp., normable). We show
that $(o)$-convergence of sequences in $LS_h(\mathcal{M})$ and
convergence in the topology $t(\mathcal{M})$ coincide if and only if
the algebra $\mathcal{M}$ is an atomic and finite algebra.

We use the von Neumann algebra terminology, notations and results
from~\cite{SZ.,Take1.}, and those that concern the theory of
measurable and locally measurable operators from~\cite{Yead1,
  Mur_m}.

\section{Preliminaries}

Let $\mathcal{H}$ be a Hilbert space over the field $\mathbf{C}$ of
complex numbers, $\mathcal{B}(\mathcal{H})$ be the $*$-algebra of all
bounded linear operators on $\mathcal{H}$, $I$ be the identity operator
on $\mathcal{H}$, $\mathcal{M}$ be a von Neumann subalgebra of
$\mathcal{B}(\mathcal{H})$, $\mathcal{P}(\mathcal{M})=\{P\in
\mathcal{M}: P^2=P=P^*\}$ be the lattice of all projections in
$\mathcal{M}$, and $\mathcal{P}_{fin}(\mathcal{M})$ be the sublattice
of its finite projections. The center of a von Neumann algebra
$\mathcal{M}$ will be denoted by $\mathcal{Z}(\mathcal{M})$.

A closed linear operator $T$ affiliated with a von Neumann algebra
$\mathcal{M}$ and having everywhere dense domain
$\mathfrak{D}(T)\subset \mathcal{H}$ is called \textit{measurable}
 if there exists a sequence
$\{P_n\}_{n=1}^{\infty}\subset \mathcal{P}(\mathcal{M})$ such that
$P_n\uparrow I$, $P_n(\mathcal{H})\subset \mathfrak{D}(T)$, and
$P_n^\bot=I-P_n \in \mathcal{P}_{fin}(\mathcal{M})$, $n=1,2,\ldots$

A set $S(\mathcal{M})$ of all measurable operators is a $*$-algebra with identity $I$ over the field
$\mathbf{C}$ ~\cite{Seg.}. It is clear
that $\mathcal{M}$ is a $*$-subalgebra of $S(\mathcal{M})$.

A closed linear operator $T$ affiliated with
$\mathcal{M}$ and having an everywhere dense domain
$\mathfrak{D}(T)\subset \mathcal{H}$ is called \textit{locally
  measurable} with respect to $\mathcal{M}$ if there is a sequence
$\{Z_n\}_{n=1}^{\infty}$ of central projections in $\mathcal{M}$
such that $Z_n\uparrow I$ and $T Z_n \in S(\mathcal{M})$ for all
$n=1,2,\ldots$

The set $LS(\mathcal{M})$ of all locally measurable operators  with
respect to $\mathcal{M}$ is a $*$-algebra with identity $I$ over the
field $\mathbf{C}$ with respect to the same algebraic operations as
in $S(\mathcal{M})$~\cite{Yead1}. Here, $S(\mathcal{M})$ is a
$*$-subalgebra of $LS(\mathcal{M})$. If $\mathcal{M}$ is  finite,
 or if $\mathcal{M}$ is a factor, the algebras $S(\mathcal{M})$
and $LS(\mathcal{M})$ coincide.

For every $T \in S(\mathcal{Z}(\mathcal{M}))$ there exists a
sequence $\{Z_n\}_{n=1}^{\infty}\subset
\mathcal{P}(\mathcal{Z}(\mathcal{M}))$ such that $Z_n \uparrow I$
and $T Z_n \in \mathcal{M}$ for all $n=1,2,\ldots$ This means that
$T\in LS(\mathcal{M})$. Hence, $S(\mathcal{Z}(\mathcal{M}))$ is a
$*$-subalgebra of $LS(\mathcal{M})$, and
$S(\mathcal{Z}(\mathcal{M}))$ coincides with the center of the
$*$-algebra $LS(\mathcal{M})$.

For every subset $E\subset LS(\mathcal{M})$, the sets of all
selfadjoint (resp., positive) operators in $E$ will be denoted by
$E_h$ (resp., $E_+$). The partial order in $LS_h(\mathcal{M})$
defined by its cone $LS_+(\mathcal{M})$ is denoted by $\leqslant$.
For a net $\{T_\alpha\}_{\alpha\in A}\subset LS_h(\mathcal{M})$, the
notation $T_\alpha\uparrow T$ (resp., $T_\alpha\downarrow T$), where
$T \in LS_h(\mathcal{M})$, means that $T_\alpha\leqslant T_\beta$
(resp., $T_\beta\leqslant T_\alpha$) for $\alpha\leqslant\beta$ and
$T=\sup\limits_{\alpha \in A}T_\alpha$ (resp.,
$T=\inf\limits_{\alpha\in A}T_\alpha$).

We say that a net $\{T_\alpha\}_{\alpha\in A}\subset
LS_h(\mathcal{M})$ \textit{$(o)$-converges} to an operator $T\in
LS_h(\mathcal{M})$, denoted by $T_\alpha
\stackrel{(o)}{\longrightarrow}T$, if there exist nets
$\{S_\alpha\}_{\alpha\in A}$ and $\{R_\alpha\}_{\alpha\in A}$ in
$LS_h(\mathcal{M})$ such that $S_\alpha \leqslant T_\alpha \leqslant
R_\alpha$ for all $\alpha\in A$ and $S_\alpha\uparrow T$,
$R_\alpha\downarrow T$.

The strongest topology on $LS_h(\mathcal{M})$ for which
$(o)$-convergence implies its convergence in the topology is called
order topology, or the $(o)$-topology, and is denoted by
$t_o(\mathcal{M})$. If $\mathcal{M}=L_\infty(\Omega,\Sigma,\mu)$,
$\mu(\Omega)<\infty$, the $(o)$-convergence of sequences in
$LS_h(\mathcal{M})$ coincides with almost everywhere convergence ,
and convergence in the $(o)$-topology, $t_o(\mathcal{M})$, with
measure convergence  ~\cite[Ch.~III,~\S\,9]{Vul.}.

Let $T$ be a closed operator with dense domain $\mathfrak{D}(T)$ in
$\mathcal{H}$, $T=U|T|$ the polar decomposition of the operator $T$,
where $|T|=(T^*T)^{\frac{1}{2}}$ and $U$ is the  partial isometry in
$\mathcal{B}(\mathcal{H})$ such that $U^*U$ is the right support of
$T$. It is known that $T\in LS(\mathcal{M})$ if and only if $|T|\in
LS(\mathcal{M})$ and $U\in \mathcal{M}$~\cite[\S\,2.3]{Mur_m}. If
$T$ is a self-adjoint operator affiliated with $\mathcal{M}$, then
the spectral family of projections $\{E_\lambda(T)\}_{\lambda\in
  \mathbf{R}}$ for $T$ belongs to
$\mathcal{M}$~\cite[\S\,2.1]{Mur_m}.

Let us now recall the definition of the locally measure
topology. Let first $\mathcal{M}$ be a commutative von
Neumann algebra. Then $\mathcal{M}$ is $*$-isomorphic to the
$*$-algebra $L_\infty(\Omega,\Sigma,\mu)$ of all essentially bounded
measurable complex-valued functions defined on a measure space
$(\Omega,\Sigma,\mu)$ with the measure $\mu$ satisfying the direct
sum property (we identify functions that are equal almost
everywhere).  The direct sum property of a measure $\mu$ means that
the Boolean algebra of all projections of the $*$-algebra
$L_\infty(\Omega,\Sigma,\mu)$ is order complete, and for any nonzero
$ P\in \mathcal{P}(\mathcal{M})$ there exists a nonzero projection
$Q\leqslant P$ such that $\mu(Q)<\infty$.

Consider the $*$-algebra
$LS(\mathcal{M})=S(\mathcal{M})=L_0(\Omega,\Sigma,\mu)$ of all
measurable almost everywhere finite complex-valued functions defined
on $(\Omega,\Sigma,\mu)$ (functions that are equal almost everywhere
are identified).  On $L_0(\Omega,\Sigma,\mu)$, define a locally measure topology
$t(\mathcal{M})$, that is, the linear Hausdorff topology, whose base of neighborhoods around zero is given by

\begin{multline*}
  W(B,\varepsilon,\delta)= \{f\in\ L_0(\Omega,\, \Sigma,\, \mu)
  \colon
  \ \hbox{there exists a set} \ E\in \Sigma\
  \mbox{such that} \\
   E\subseteq B, \ \mu(B\setminus
  E)\leqslant\delta, \ f\chi_E \in L_\infty(\Omega,\Sigma,\mu), \
  \|f\chi_E\|_{{L_\infty}(\Omega,\Sigma,\mu)}\leqslant\varepsilon\},
\end{multline*}
where $\varepsilon, \ \delta >0$, $B\in\Sigma$, $\mu(B)<\infty$, and
$\chi(\omega)=\begin{cases} 1, \ \ \ \omega\in E, \\ 0, \ \ \ \omega
  \ \not\in \ E. \end{cases}$

Convergence of a net $\{f_\alpha\}$ to $f$ in the topology
$t(\mathcal{M})$, denoted by $f_\alpha
\stackrel{t(\mathcal{M})}{\longrightarrow}f$, means that $f_\alpha
\chi_B \longrightarrow f\chi_B$ in measure $\mu$ for any $B\in
\Sigma$ with $\mu(B)<\infty$. It is clear that the topology
$t(\mathcal{M})$ does not change if the measure $\mu$ is replaced
with an equivalent measure.
Denote by $t_h(\mathcal{M})$ the topology on $LS_h(\mathcal{M})$
induced by the topology $t(\mathcal{M})$ on $LS(\mathcal{M})$.

\begin{prop}
\label{p_o=tl}
If $\mathcal{M}$ is a commutative von Neumann algebra, then
$$
t_h(\mathcal{M})\leqslant t_o(\mathcal{M}).
$$
\end{prop}

\begin{proof}
  It sufficient to prove that any net $\{f_\alpha\}_{\alpha\in
    A}\subset LS_h(\mathcal{M})$, which $(o)$-converges to zero,
  also converges to zero with respect to the topology
  $t_h(\mathcal{M})$. Choose a net $\{g_\alpha\}_{\alpha\in
    A}\subset LS_h(\mathcal{M})$ such that $g_\alpha\downarrow 0$
  and $-g_\alpha \leqslant f_\alpha\leqslant g_\alpha$ for all
  $\alpha \in A$.

  Let $B \in \Sigma$ and $\mu(B)<\infty$ (we identify $\mathcal{M}$
  with $L_\infty(\Omega,\Sigma,\mu)$). Then
$$
-g_\alpha\chi_B \leqslant f_\alpha\chi_B\leqslant g_\alpha\chi_B, \
\alpha\in A,
$$
and, since $g_\alpha\chi_B\downarrow 0$, we have $g_\alpha
\chi_B\rightarrow 0$ in measure $\mu$. Consequently,
$f_\alpha\chi_B\rightarrow 0$ in measure $\mu$ and, hence, $f_\alpha
\stackrel{t_h(\mathcal{M})}{\longrightarrow}0$.
\end{proof}

Let now $\mathcal{M}$ be an arbitrary von Neumann algebra. Identify
the center $\mathcal{Z}(\mathcal{M})$ with the $*$-algebra
$L_\infty(\Omega,\Sigma,\mu)$, and $LS(\mathcal{Z}(\mathcal{M}))$
with the $*$-algebra $L_0(\Omega,\Sigma,\mu)$. Denote by
$L_+(\Omega,\, \Sigma,\, m)$ the set of all measurable real-valued
functions defined on $(\Omega,\Sigma,\mu)$ and taking values in the
extended half-line $[0,\, \infty]$ (functions that are equal almost
everywhere are identifed). It was shown
in~\cite{Seg.} that there exists a mapping $\mathcal{D}\colon
\mathcal{P}(\mathcal{M})\to L_+(\Omega,\Sigma,\mu)$ that possesses
the following properties:
\begin{itemize}
\item[(i)] $\mathcal{D}(P)=0$ if and only if $P=0$;
\item[(ii)]  $\mathcal{D}(P)\in L_0(\Omega,\Sigma,\mu)\Longleftrightarrow P\in \mathcal{P}_{fin}(\mathcal{M})$;
\item[(iii)] $\mathcal{D}(P\vee Q)=\mathcal{D}(P)+\mathcal{D}(Q)$ if
  $PQ=0$;
\item[(iv)] $\mathcal{D}(U^*U)=\mathcal{D}(UU^*)$ for any partial
  isometry $U\in \mathcal{M}$;
\item[(v)] $\mathcal{D}(ZP)=Z\mathcal{D}(P)$ for any $Z\in
  \mathcal{P}(\mathcal{Z}(\mathcal{M}))$ and $P\in
  \mathcal{P}(\mathcal{M})$;
\item[(vi)] if $\{P_\alpha\}_{\alpha\in A}, P\in
  \mathcal{P}(\mathcal{M})$ and $P_\alpha\uparrow P$, then
  $\mathcal{D}(P)=\sup\limits_{\alpha\in A}\mathcal{D}(P_\alpha)$.
\end{itemize}

A mapping $\mathcal{D}\colon \mathcal{P}(\mathcal{M})\to
L_+(\Omega,\Sigma,\mu)$ that satisfies properties (i)---(vi) is
called a \textit{dimension function} on $\mathcal{P}(\mathcal{M})$.

For arbitrary numbers $\varepsilon , \delta >0$ and a set $B\in
\Sigma$, $\mu(B)<\infty$, set
\begin{multline*}
  V(B,\varepsilon, \delta ) = \{T\in LS(\mathcal{M})\colon \
  \mbox{there exist} \ P\in \mathcal{P}(\mathcal{M}),\
  Z\in \mathcal{P}(\mathcal{Z}(\mathcal{M})),\\
   \mbox{such that} \ TP\in \mathcal{M},
  \|TP\|_{\mathcal{M}}\leqslant\varepsilon,
  \ Z^\bot \in W(B,\varepsilon,\delta), \
    \mathcal{D}(ZP^\bot)\leqslant\varepsilon Z\},
\end{multline*}
where $\|\cdot\|_{\mathcal{M}}$ is the $C^*$-norm on $\mathcal{M}$.

It was shown in~\cite{Yead1} that the system of sets
 $$
 \{\{T+V(B,\,\varepsilon,\,\delta)\}\colon \ T \in LS(\mathcal{M}),\
 \varepsilon, \ \delta >0,\ B\in\Sigma,\ \mu(B)<\infty\} \leqno (1)
$$
defines a linear Hausdorff  topology $t(\mathcal{M})$ on
$LS(\mathcal{M})$ such that sets~(1) form a neighborhood base of the
operator $T\in LS(\mathcal{M})$. Here, $(LS(\mathcal{M}),
t(\mathcal{M}))$ is a complete topological $*$-algebra, and the
topology $t(\mathcal{M})$ does not depend on a choice of the
dimension function $\mathcal{D}$.

The topology $t(\mathcal{M})$ is called a \textit{locally
  measure topology }~\cite{Yead1}.

We will need the following  criterion for convergence of nets
with respect to this topology.

\begin{prop}[{\cite[\S\,3.5]{Mur_m}}]
  \label{plm-spk1}
  \begin{itemize}
  \item[(i)] A net $\{P_\alpha\}_{\alpha\in A}\subset
    \mathcal{P}(\mathcal{M})$ converges to zero with respect to the
    topology $t(\mathcal{M})$ if and only if there is a net
    $\{Z_\alpha\}_{\alpha\in A}\subset
    \mathcal{P}(\mathcal{Z}(\mathcal{M}))$ such that $Z_\alpha
    P_\alpha\in \mathcal{P}_{fin}(\mathcal{M})$ for all $\alpha\in
    A$, $Z^\bot_\alpha
    \stackrel{t(\mathcal{Z}(\mathcal{M}))}{\longrightarrow} 0$, and
    $\mathcal{D}(Z_\alpha
    P_\alpha)\stackrel{t(\mathcal{Z}(\mathcal{M}))}{\longrightarrow}
    0$, where $t(\mathcal{Z}(\mathcal{M}))$ is the locally
    measure topology  on $LS(\mathcal{Z}(\mathcal{M}))$.
  \item[(ii)] A net $\{T_\alpha\}_{\alpha\in A} \subset
    LS(\mathcal{M})$ converges to zero with respect to the topology
    $t(\mathcal{M})$ if and only if $E^\bot_\lambda(|T_\alpha|)
    \stackrel{t(\mathcal{M})}{\longrightarrow} 0$ for any
    $\lambda>0$, where $\{E^\bot_\lambda(|T_\alpha|)\}$ is a
    spectral projection family for the operator $|T_\alpha|$.
  \end{itemize}
\end{prop}

It follows from Proposition~\ref{plm-spk1} that the topology
$t(\mathcal{M})$ induces the topology $t(\mathcal{Z}(\mathcal{M}))$
on $LS(\mathcal{Z}(\mathcal{M}))$; hence,
$S(\mathcal{Z}(\mathcal{M}))$ is a closed $*$-subalgebra of
$(LS(\mathcal{M}), t(\mathcal{M}))$.

It is clear that
$$
X\cdot V(B,\varepsilon,\delta)\subset V(B,\varepsilon,\delta)
$$
for any $X\in \mathcal{M}$ with the norm
$\|X\|_{\mathcal{M}}\leqslant 1$. Since
$V^*(B,\varepsilon,\delta)\subset
V(B,2\varepsilon,\delta)$~\cite[\S\,3.5]{Mur_m}, we have
$$
V(B,\varepsilon,\delta)\cdot Y \subset V(B,4\varepsilon, \delta)
$$
for all $Y\in \mathcal{M}$ satisfying $\|Y\|_{\mathcal{M}}\leqslant
1$. Hence,
$$
X\cdot V(B,\varepsilon,\delta)\cdot Y \subset V(B,4\varepsilon,
\delta) \leqno (2)
$$
for any $\varepsilon,\delta>0, \ B\in \Sigma$, $\mu(B)<\infty$, $X,Y
\in \mathcal{M}$ with $\|X\|_{\mathcal{M}}\leqslant 1$,
$\|Y\|_{\mathcal{M}}\leqslant 1$.

Since the involution is continuous in the topology $t(\mathcal{M})$,
the set $LS_h(\mathcal{M})$ is closed in
$(LS(\mathcal{M}),t(\mathcal{M}))$. The cone $LS_+(\mathcal{M})$ of
positive elements is also closed in
$(LS(\mathcal{M}),t(\mathcal{M}))$~\cite{Yead1}. Hence, for every
increasing (or decreasing) net $\{T_\alpha \}_{\alpha\in A}\subset
LS_h(\mathcal{M})$ that converges to $T$ in the topology
$t(\mathcal{M})$, we have that $T\in LS_h(\mathcal{M})$ and
$T=\sup\limits_{\alpha\in A}T_\alpha$ (resp.
$T=\inf\limits_{\alpha\in A}T_\alpha$)~\cite[Ch.~V, \S\,4]{Shef.}.

\section{Comparison of the topologies  $t(\mathcal{M})$ and $t_o(\mathcal{M})$}

Let $\mathcal{M}$ be an arbitrary von Neumann algebra,
$t_o(\mathcal{M})$ be the $(o)$-topology on $LS_h(\mathcal{M})$. As
before, $t_h(\mathcal{M})$ denotes the topology on
$LS_h(\mathcal{M})$ induced by the topology $t(\mathcal{M})$ on
$LS(\mathcal{M})$.

\begin{thm}
\label{t_(o)=tl}
The following conditions are equivalent:
\begin{itemize}
\item[(i)] $t_h(\mathcal{M})\leqslant t_o(\mathcal{M})$;
\item[(ii)] $\mathcal{M}$ is finite.
\end{itemize}
\end{thm}

\begin{proof}
  $(i)\Rightarrow (ii)$. Suppose that $\mathcal{M}$ is not finite. Then there is a sequence of pairwise
  orthogonal and pairwise equivalent projections
  $\{P_n\}_{n=1}^\infty$ in $\mathcal{P}(\mathcal{M})$. Choose a
  partial isometry $U_n$ in $\mathcal{M}$ such that $U_n^*U_n=P_1, \
  \ U_nU_n^*=P_n, \ \ n=1,2,\ldots$ Set $Q_n=\sup\limits_{j\geqslant
    n}P_j$. Then $Q_n\in \mathcal{P}(\mathcal{M})$ and
  $Q_n\downarrow 0$. By condition $(i)$ we have $Q_n
  \stackrel{t_h(\mathcal{M})}{\longrightarrow} 0$. Since
  $P_n=P_nQ_n$, it follows from~$(2)$ that $P_n
  \stackrel{t_h(\mathcal{M})}{\longrightarrow} 0$. Again using $(2)$
  we get that $P_1=U_n^*P_nU_n
  \stackrel{t_h(\mathcal{M})}{\longrightarrow} 0$ and, hence,
  $P_1=0$, which is not true. Consequently, $\mathcal{M}$ is
  finite.

  $(ii)\Rightarrow (i)$. Let $\mathcal{M}$ be a finite von
  Neumann algebra, $\Phi: \mathcal{M}\mapsto
  \mathcal{Z}(\mathcal{M})$ a center-valued trace on
  $\mathcal{M}$~\cite[Ch.~V, \S\,2]{Take1.}. The restriction
  $\mathcal{D}$ of the trace $\Phi$ on $\mathcal{P}(\mathcal{M})$ is
  a dimension function on $\mathcal{P}(\mathcal{M})$. Let
  $\{T_\alpha\}_{\alpha\in A}\subset LS_h(\mathcal{M})$ and
  $T_\alpha \stackrel{(o)}{\longrightarrow}0$. Then there exists a
  net $\{S_\alpha\}_{\alpha\in A}$ in $LS_h(\mathcal{M})$ such that
  $S_\alpha\downarrow 0$ and $-S_\alpha \leqslant T_\alpha\leqslant
  S_\alpha$ for all $\alpha\in A$. Fix $\alpha_0\in A$ and set
  $X_\alpha=XT_\alpha X, \ Y_\alpha=X S_\alpha X$ for
  $\alpha\geqslant\alpha_0$, where $
  X=(I+S_{\alpha_0})^{-\frac{1}{2}} $. It is clear that $-I\leqslant
  -Y_\alpha\leqslant X_\alpha\leqslant Y_\alpha \leqslant I$ for
  $\alpha\geqslant\alpha_0$ and $Y_\alpha\downarrow 0$.
  Consequently, $-I\leqslant -\Phi(Y_\alpha)\leqslant
  \Phi(X_\alpha)\leqslant \Phi(Y_\alpha) \leqslant I$ and
  $\Phi(Y_\alpha)\downarrow 0$.

  Let $E_\lambda^\bot(Y_\alpha)=\{Y_\alpha > \lambda \}$ be a
  spectral projection for $Y_\alpha$ corresponding to the interval
  $(\lambda,+\infty), \ \lambda>0$. Since
$$
\mathcal{D}(E_\lambda^\bot(Y_\alpha))\leqslant
\frac{1}{\lambda}\Phi(Y_\alpha),
$$
it follows that $\mathcal{D}(E_\lambda^\bot(Y_\alpha))
\stackrel{(o)}{\longrightarrow}0$ in $\mathcal{Z}(\mathcal{M})$. By
Proposition~\ref{p_o=tl}, we have that
$$
\mathcal{D}(E_\lambda^\bot(Y_\alpha))
\stackrel{t(\mathcal{Z}(\mathcal{M}))}{\longrightarrow}0
$$
for all $\lambda>0$. Hence, Proposition~\ref{plm-spk1} gives that
$Y_\alpha \stackrel{t(\mathcal{M})}{\longrightarrow}0$.

Set $Z_\alpha=X_\alpha+Y_\alpha$. Repeating the previous reasoning
and using the inequality $0\leqslant Z_\alpha\leqslant 2Y_\alpha$ we
get that $Z_\alpha \stackrel{t(\mathcal{M})}{\longrightarrow}0$.
Consequently, $X_\alpha=Z_\alpha-Y_\alpha
\stackrel{t(\mathcal{M})}{\longrightarrow}0$ and, hence,
$T_\alpha=X^{-1}X_\alpha X^{-1}
\stackrel{t(\mathcal{M})}{\longrightarrow}0$. Thus,
$t_h(\mathcal{M})\leqslant t_o(\mathcal{M})$.
\end{proof}

\begin{note}
\label{r_(o)=tl}
In the proof of the implication $(i)\Rightarrow (ii)$ of
Theorem~\ref{t_(o)=tl}, it was shown that convergence to zero, in
the topology $t(\mathcal{M})$, of any sequence of projections in
$\mathcal{P}(\mathcal{M})$, which decreases to zero, implies that
$\mathcal{M}$ is finite.
\end{note}

Let us now find conditions that would imply that the topologies
$t_h(\mathcal{M})$ and $t_o(\mathcal{M})$ coincide on
$LS_h(\mathcal{M})$. Recall that a von Neumann algebra $\mathcal{M}$
is called $\sigma$-finite if any family of nonzero mutually
orthogonal projections in $\mathcal{P}(\mathcal{M})$ is at most
countable. It is known that the topology $t(\mathcal{M})$ on
$LS(\mathcal{M})$ is metrizable if and only if the center
$\mathcal{Z}(\mathcal{M})$ is $\sigma$-finite~\cite{Yead1}.

\begin{prop}
\label{p_tl>(o)}
If $\mathcal{Z}(\mathcal{M})$ is $\sigma$-finite, then
$t_o(\mathcal{M})\leqslant t_h(\mathcal{M}).$
\end{prop}

\begin{proof}
  Choose a neighborhood basis $\{V_k\}_{k=1}^\infty$ of zero in
  $(LS(\mathcal{M}), t(\mathcal{M}))$ such that
$$
V_{k+1}+V_{k+1}\subset V_k
$$
for all $k$.

Let $\{T_n\}_{n=1}^\infty\subset LS_h(\mathcal{M})$ and $T_n
\stackrel{t(\mathcal{M})}{\longrightarrow}0$. Using relation~$(2)$
and the polar decomposition $T_n=U_n|T_n|$ we see that $|T_n|
\stackrel{t(\mathcal{M})}{\longrightarrow}0$. Choose a subsequence
$|T_{n_k}|\in V_k$ and set $S_k=\sum\limits_{i=1}^k |T_{n_i}|$. It
is clear that $S_m-S_{k+1} \in V_k$ for $m>k$. Hence, there exists
an operator $S\in LS_h(\mathcal{M})$ such that $S_k
\stackrel{t(\mathcal{M})}{\longrightarrow}S$. The sequence
$R_k=S-\sum\limits_{i=1}^{k}|T_{n_i}|$ decreases and $R_k
\stackrel{t(\mathcal{M})}{\longrightarrow}0$. Since the cone
$LS_+(\mathcal{M})$ of positive elements is closed in
$(LS(\mathcal{M}),t(\mathcal{M}))$, we have $R_k\downarrow 0$ and
$$
-R_{k-1} \leqslant -|T_{n_k}| \leqslant T_{n_k} \leqslant |T_{n_k}|
\leqslant R_{k-1}.
 $$
 Consequently, $T_{n_k} \stackrel{(o)}{\longrightarrow}0$. Thus, for
 any sequence $\{T_n\}_{n=1}^\infty \subset LS_h(\mathcal{M})$,
 which converges to $T\in LS_h(\mathcal{M})$ in the topology
 $t(\mathcal{M})$, there exists a subsequence $T_{n_k}
 \stackrel{(o)}{\longrightarrow}T$. This means that
 $t_o(\mathcal{M})\leqslant t_h(\mathcal{M})$.
\end{proof}

We now describe a class of von Neumann algebras $\mathcal{M}$ for
which the topologies $t_o(\mathcal{M})$ and $t_h(\mathcal{M})$ coincide.

\begin{thm}
\label{t_t0=tl}
The following conditions are equivalent:
\begin{itemize}
\item [(i)] $\mathcal{M}$ is finite and $\sigma$-finite;
\item [(ii)] $t_o(\mathcal{M})=t_h(\mathcal{M})$.
\end{itemize}
\end{thm}

\begin{proof}
  The implication $(i)\Rightarrow (ii)$ follows from
  Theorem~\ref{t_(o)=tl} and Proposition~\ref{p_tl>(o)}.

  $(ii)\Rightarrow (i)$. If $t_o(\mathcal{M})=t_h(\mathcal{M})$,
  then the von Neumann algebra $\mathcal{M}$ is finite by
  Theorem~\ref{t_(o)=tl}. Let us show that the center
  $\mathcal{Z}(\mathcal{M})$ is $\sigma$-finite.

  Let $\{Z_j\}_{j\in \Delta}$ be a family of nonzero pairwise
  orthogonal projections in $\mathcal{P}(\mathcal{Z}(\mathcal{M}))$
  satisfying $\sup\limits _{j\in \Delta}Z_j=I$ and $\mu(Z_j)<\infty$
  (as before, we identify the commutative von Neumann algebra
  $\mathcal{Z}(\mathcal{M})$ with $L_\infty(\Omega,\Sigma,\mu)$ and
  $LS(\mathcal{Z}(\mathcal{M}))$ with $L_0(\Omega,\Sigma,\mu)$).
  Denote by $E$ a $*$-subalgebra in $L_0(\Omega,\Sigma,\mu)$ of all
  functions $f\in L_0(\Omega,\Sigma,\mu)$ satisfying
  $fZ_j=\lambda_jZ_j$ for some $\lambda_j\in \mathbf{C}, \ j\in
  \Delta$. It is clear that $E$ is $*$-isomorphic to the $*$-algebra
  $\mathbf{C}^\Delta=\{\{\lambda_j\}_{j\in \Delta}: \lambda_j \in
  \mathbf{C}\}$, and $E_h$ is isomorphic to the algebra
  $\mathbf{R}^\Delta=\{\{r_j\}_{j\in \Delta}: r_j \in \mathbf{R}\}$.
  Denote by $t$ the Tychonoff topology of coordinate convergence in
  $\mathbf{R}^\Delta$, and identify $E_h$ with $\mathbf{R}^\Delta$.
  If $f_\alpha \in E_h$, then $f_\alpha
  \stackrel{t(\mathcal{M})}{\longrightarrow}0$ if and only if
  $f_\alpha Z_j \stackrel{\mu}{\longrightarrow}0$ for all
  $j\in\Delta$. This means that the topology $t(\mathcal{M})$
  induces the Tychonoff topology $t$ on $E_h$.

  Let us show that any subset $G\subset E_h$, upper bounded in
  $LS_h(\mathcal{M})$, is upper bounded in $E_h$, and the least
  upper bounds for $G$ in $E_h$ and in
  $S_h(\mathcal{M})=LS_h(\mathcal{M})$ are the same.

  For any operator $T\in S_+(\mathcal{M})$ there exists a maximal
  commutative \hbox{$*$-sub}algebra $\mathcal{A}$ of
  $S(\mathcal{M})$ containing $\mathcal{Z}(\mathcal{M})$ and $T$.
  Since $\mathcal{M}$ is a finite von Neumann algebra,
  $\mathcal{N}=\mathcal{A}\cap \mathcal{M}$ is also a finite von
  Neumann algebra, and $\mathcal{A}=S(\mathcal{N})$. We also have
  that $\mathcal{Z}(\mathcal{M})\subset \mathcal{N}$. It is clear
  that $S_h(\mathcal{Z}(\mathcal{M}))$ is a regular sublattice of
  $S_h(\mathcal{N})$, that is, the least upper bounds and the least
  lower bounds of bounded subsets of $S_h(\mathcal{Z}(\mathcal{M}))$
  calculated in $S_h(\mathcal{N})$ and in
  $S_h(\mathcal{Z}(\mathcal{M}))$ coincide.

  Let $G\subset E_h$ and $S\leqslant T$ for all $S\in G$. Then there
  exists a least upper bound $\sup G$ in $S_h(\mathcal{N})$, which,
  since $S_h(\mathcal{Z}(\mathcal{M}))$ is regular, belongs to
  $S_h(\mathcal{Z}(\mathcal{M}))$. Since $E_h$ is a regular
  sublattice in $S_h(\mathcal{Z}(\mathcal{M}))$, $\sup G\in E_h$.
  Consequently, any net $\{S_\alpha\}\subset E_h$ that
  $(o)$-converges to $S$ in $S_h(\mathcal{M})$ will be
  $(o)$-convergent to $S$ in $E_h$. This means that the
  $(o)$-topology $t_o(\mathcal{M})$ in $S_h(\mathcal{M})$ induces
  the $(o)$-topology $t_o(E_h)$ in $E_h$. Since
  $t_o(\mathcal{M})=t_h(\mathcal{M})$, the Tychonoff topology $t$
  coincides with the $(o)$-topology in $\mathbf{R}^\Delta$.
  Consequently, the set $\Delta$ is at most countable~\cite[Ch.~V,
  \S\,6]{Sar.}, that is, $\mathcal{Z}(\mathcal{M})$ is a
  $\sigma$-finite von Neumann algebra. Since the von Neumann algebra
  $\mathcal{M}$ is finite, $\mathcal{M}$ is also a
  $\sigma$-finite algebra~\cite{Seg.}.
\end{proof}

\medskip

Proposition~\ref{p_tl>(o)} and Theorems~\ref{t_(o)=tl} and
~\ref{t_t0=tl} give the following.

\begin{corr}
\label{c_(o)<tl}
\begin{itemize}
\item [$(i)$] If $\mathcal{M}$ is a $\sigma$-finite von Neumann
  algebra but is not finite, then $t_o(\mathcal{M})<
  t_h(\mathcal{M})$.
\item [$(ii)$] If $\mathcal{M}$ is not a $\sigma$-finite von Neumann
  algebra but is finite, then
  $t_h(\mathcal{M})<t_o(\mathcal{M})$.
\end{itemize}
\end{corr}

Using Corollary~\ref{c_(o)<tl} one can easily construct an example
of a von Neumann algebra $\mathcal{M}$ for which the topologies
$t_o(\mathcal{M})$ and $t_h(\mathcal{M})$ are incomparable.

Let $\mathcal{M}_1$ be a $\sigma$-finite von Neumann algebra which
is not finite, $\mathcal{M}_2$ be a not $\sigma$-finite von
Neumann algebra which is finite, and $\mathcal{M}=\mathcal{M}_1
\times\mathcal{M}_2$. Then $ LS(\mathcal{M})=LS(\mathcal{M}_1)\times
LS(\mathcal{M}_2)$~\cite[\S\,2.5]{Mur_m}, and the topology
$t(\mathcal{M})$ coincides with the product of the topologies
$t(\mathcal{M}_1)$ and $t(\mathcal{M}_2)$. Moreover, a net
$\{(T^{(1)}_\alpha, T^{(2)}_\alpha)\}_{\alpha\in A}$ in
$LS_h(\mathcal{M}_1)\times LS_h(\mathcal{M}_2)$ is
$(o)$-convergent to an element $(T^{(1)}, T^{(2)})\in
LS_h(\mathcal{M}_1)\times LS_h(\mathcal{M}_2)$ if and only if the
net $\{T^{(k)}_\alpha\}_{\alpha\in A}$ is $(o)$-convergent to
$T^{(k)}$, $k=1,2$. Identifying $LS_h(\mathcal{M}_1)$ with the
linear subspace $LS_h(\mathcal{M}_1)\times \{0\}$ and
$LS_h(\mathcal{M}_2)$ with $\{0\} \times LS_h(\mathcal{M}_2)$ we get
that the $(o)$-topology $t_o(\mathcal{M})$ in $LS_h(\mathcal{M})$
induces $(o)$-topologies in $LS_h(\mathcal{M}_1)$ and
$LS_h(\mathcal{M}_2)$, correspondingly.  It remains to apply
Corollary~\ref{c_(o)<tl}, by which the topologies $t_o(\mathcal{M})$
and $t_h(\mathcal{M})$ are incomparable.

\section{Locally measure topology  on semifinite von Neumann algebras}

Let $\mathcal{M}$ be a semifinite von Neumann algebra acting on a
Hilbert space $\mathcal{H}$, $\tau$ be a faithful normal semifinite
trace on $\mathcal{M}$. An operator $T\in S(\mathcal{M})$ with
domain $\mathfrak{D}(T)$ is called \textit{$\tau$-measurable} if for
any $\varepsilon>0$ there exists a projection $P\in
\mathcal{P}(\mathcal{M})$ such that $P(\mathcal{H})\subset
\mathfrak{D}(T)$ and $\tau(P^\bot)<\varepsilon$.

A set $S(\mathcal{M},\tau)$ of all $\tau$-measurable operators is
a $*$-subalgebra of $S(\mathcal{M})$, and $\mathcal{M}\subset
S(\mathcal{M},\tau)$. If the trace $\tau$ is finite, then
$S(\mathcal{M},\tau)=S(\mathcal{M})$.

Let $t_\tau$ be \textit{a measure topology}~\cite{Nels.} on $S(\mathcal{M},\tau)$
whose base of neighborhoods around zero is given by
\begin{multline*}
  U(\varepsilon,\delta)=\{T\in S(\mathcal{M},\tau): \ \
  \hbox{there exists a projection} \ \
  P\in \mathcal{P}(\mathcal{M}),\\
  \hbox{such that}\ \tau(P^\bot)\leqslant \delta,
  \ TP \in \mathcal{M}, \ \ \|TP\|_\mathcal{M}\leqslant\varepsilon\}, \ \ \varepsilon>0, \ \delta>0.
\end{multline*}

The pair $(S(\mathcal{M},\tau),t_\tau)$ is a complete metrizable
topological $*$-algebra. Here, the topology $t_\tau$ majorizes the
topology $t(\mathcal{M})$ on $S(\mathcal{M},\tau)$ and, if $\tau$ is
a finite trace, the topologies $t_\tau$ and $t(\mathcal{M})$
coincide~\cite[\S\S\,3.4, 3.5]{Mur_m}. Denote by
$t(\mathcal{M},\tau)$ the topology on $S(\mathcal{M},\tau)$ induced
by the topology $t(\mathcal{M})$. It is not
true in general that, if the topologies $t_\tau$ and
$t(\mathcal{M},\tau)$ are the same, then the von Neumann algebra
$\mathcal{M}$ is finite.
Indeed, if $\mathcal{M}=\mathcal{B}(\mathcal{H}), \
\dim(\mathcal{H})=\infty$, $\tau=tr$ is the canonical trace on
$\mathcal{B}(\mathcal{H})$, then
$LS(\mathcal{M})=S(\mathcal{M})=S(\mathcal{M},\tau)=\mathcal{M}$,
and the two topologies $t_\tau$ and $t(\mathcal{M})$ coincide with
the uniform topology on $\mathcal{B}(\mathcal{H})$.

At the same time, we have the following.
\begin{prop}
  \label{p_t-kon}
  If $\mathcal{M}$ is a finite von Neumann algebra with a faithful
  normal semifinite trace $\tau$ and $t_\tau = t(\mathcal{M},
  \tau)$, then $\tau(I)<\infty$.
\end{prop}

\begin{proof}
  If $\tau(I)=\infty$, then there exists a sequence of projections
  $\{P_n\}_{n=1}^\infty \subset \mathcal{P}(\mathcal{M})$ such that
  $P_n\downarrow 0$ and $\tau(P_n)=\infty$. By
  Theorem~\ref{t_(o)=tl}, $P_n
  \stackrel{t(\mathcal{M})}{\longrightarrow}0$, however,
  $\{P_n\}_{n=1}^\infty$ does not converge to zero in the topology
  $t_\tau$.
\end{proof}

Denote by $t_{h\tau}$ the topology on $S_h(\mathcal{M}, \tau)$
induced by the topology $t_\tau$, and by $t_{o\tau}(\mathcal{M})$
the $(o)$-topology on $S_h(\mathcal{M},\tau)$. The topology
$t_{o\tau}(\mathcal{M})$, in general, does not coincide with the
topology induced by the $(o)$-topology $t_{o}(\mathcal{M})$ on
$S_h(\mathcal{M},\tau)$. For example, for
$$
\mathcal{M}=l_\infty (\mathbf{C})=\{\{\alpha_n\}_{n=1}^\infty
\subset \mathbf{C}: \sup\limits_{n\geqslant 1}|\alpha_n|<\infty\}
$$
and
$$
\tau(\{\alpha_n\})=\sum\limits_{n=1}^\infty \alpha_n, \ \ \
\alpha_n\geqslant 0,
$$
we have that $LS(\mathcal{M})=\mathbf{C}^{\mathbf{N}}$ and
$S(\mathcal{M},\tau)=l_\infty(\mathbf{C})$. Here, the $(o)$-topology
$t_o(\mathcal{M})$ on $LS_h(\mathcal{M})=\mathbf{R}^{\Delta}$ is the
topology of the coordinatewise convergence, in particular,
$$
T_n=n\, Z_n\stackrel{(o)}{\longrightarrow}0
$$
in $\mathbf{R}^\Delta$, where $Z_n=\{0,\ldots,0, 1, 0,\ldots\}$, the
number $1$ is at the $n$-th place. However, the sequence
$\{T_n\}_{n=1}^\infty$ does not converge in the
\hbox{$(o)$-topo}logy $t_{o \tau}(l_\infty(\mathbf{C}))$, since any
its subsequence is not bounded in
$l_\infty(\mathbf{R})=(l_\infty(\mathbf{C}))_h$~\cite[Ch.~VI,
$\S\,3$]{Vul.}.

\begin{note}
  \label{r_por_1}
  Since $S_+(\mathcal{M},\tau)=S(\mathcal{M},\tau)\cap
  LS_+(\mathcal{M})$ is closed in $(S(\mathcal{M},\tau),t_\tau)$ and
  $T_n \stackrel{t_\tau}{\longrightarrow}T$ if and only if $|T_n-T|
  \stackrel{t_\tau}{\longrightarrow}0$~\cite[\S\,3.4]{Mur_m}, using
  metrizability of the topology $t_\tau$ and repeating the end of
  the proof of Proposition~\ref{p_tl>(o)} we get that
  $t_{o\tau}(\mathcal{M})\leqslant t_{h\tau}$.
\end{note}

\begin{note}
\label{r_por_2}
Using the inclusions $U^*(\varepsilon,\delta)\subset
U(\varepsilon,2\delta)$ and $TU(\varepsilon,\delta)\subset
U(\varepsilon \|T\|_\mathcal{M},\delta)$, where $T\in \mathcal{M}$
we can see as in the proof of the implication $(i)\Rightarrow(ii)$
in Theorem~\ref{t_(o)=tl} that the equality
$t_{o\tau}(\mathcal{M})=t_{h\tau}$ implies that the von Neumann
algebra $\mathcal{M}$ is finite.
\end{note}

\begin{prop}\label{p_(o)=ttau}
  Let $\mathcal{M}$ be a semifinite von Neumann algebra, $\tau$ be a
  faithful normal semifinite trace on $\mathcal{M}$. The following
  conditions are equivalent.
  \begin{itemize}
  \item [$(i)$] Any net that is $(o)$-convergent in
    $S_h(\mathcal{M},\tau)$ also converges in the topology
    $t_{h\tau}$;
  \item[$(ii)$] $t_{o\tau}(\mathcal{M})=t_{h\tau}$;
  \item[$(iii)$] $\tau(I)<\infty$.
  \end{itemize}
\end{prop}

\begin{proof}
  The implication $(i)\Rightarrow (ii)$ follows from
  Remark~\ref{r_por_1} and definition of the topology
  $t_{o\tau}(\mathcal{M})$.

  $(ii)\Rightarrow (iii)$. By Remark~\ref{r_por_2}, the von Neumann
  algebra $\mathcal{M}$ is finite. Repeating the proof of
  Proposition~\ref{p_t-kon} we see that $\tau(I)<\infty$.

  $(iii)\Rightarrow (i)$. If $\tau(I)<\infty$, then $\mathcal{M}$ is
  a finite von Neumann algebra, and
  $S(\mathcal{M},\tau)=LS(\mathcal{M})$. Hence, $(i)$ follows from
  Theorem~\ref{t_(o)=tl}.
\end{proof}

Together with the topology $t_\tau$ on $S(\mathcal{M},\tau)$, one
can also consider two more Hausdorff vector topologies associated
with the trace $\tau$~\cite{Bik2}. This are \textit{the $\tau$-locally measure topology $t_{\tau l}$}
 and \textit{the weak $\tau$-locally measure topology $t_{w \tau l}$}.
 The sets
$$
U_\tau(\varepsilon,\delta, P)=\{T\in S(\mathcal{M},\tau): \ \
\hbox{there exists a projection} \ \ Q\in \mathcal{P}(\mathcal{M})
$$
$$
\ \ \hbox{such that}\quad Q\leqslant P, \, \tau(P-Q)\leqslant
\delta, \ TQ \in \mathcal{M}, \ \
\|TQ\|_\mathcal{M}\leqslant\varepsilon\}
$$
(resp.,
$$
U_{w \tau}(\varepsilon,\delta,P)=\{T\in S(\mathcal{M},\tau): \ \
\hbox{there exists a projection} \ \ Q\in \mathcal{P}(\mathcal{M})
$$
$$
 \hbox{such that}\quad
Q\leqslant P, \, \tau(P-Q)\leqslant \delta, \ Q\,TQ \in \mathcal{M}, \ \ \|Q\,TQ\|_\mathcal{M}\leqslant\varepsilon\}),
$$
where $\varepsilon>0, \ \delta>0, \ P\in \mathcal{P}(\mathcal{M}), \
\tau(P)<\infty$, form a neighborhood base around in the topology
$t_{\tau l}$ (resp., in the topology $t_{w \tau l}$).

It is clear that $t_{w \tau l}\leqslant t_{\tau l} \leqslant
t_\tau$, and if $\tau(I)<\infty$, all three topologies $t_{w \tau
  l}$, $t_{\tau l}$, and $t_\tau$ coincide.

\medskip

Let us remark that if $\mathcal{M}=\mathcal{B}(\mathcal{H})$ and
$\tau=tr$, the topology $t_{\tau l}$ coincides with the strong
operator topology, and the topology $t_{w \tau l}$ with the weak
operator topology, that is, if $\dim(\mathcal{H})=\infty$, we have
$t_{w \tau l} < t_{\tau l} < t_\tau$ in this case.

The following criterion for convergence of nets in the topologies $t_{w \tau l}$ and $t_{\tau l}$
 can be obtained directly from the definition.

\begin{prop}[\cite{Bik2}]\label{p_red_top}
  If $\{T_\alpha\}_{\alpha \in A}, T \subset S(\mathcal{M},\tau)$,
  then $T_\alpha \stackrel{t_{\tau l}} {\longrightarrow}T$ (resp.,
  $T_\alpha \stackrel{t_{w \tau l}}{\longrightarrow}T$) if and only
  if $T_\alpha P\stackrel{t_{\tau}}{\longrightarrow}TP$ (resp.,
  $PT_\alpha P\stackrel{t_{\tau}} {\longrightarrow}PTP$) for all
  $P\in \mathcal{P}(\mathcal{M})$ satisfying $\tau(P)<\infty$.
\end{prop}

\medskip Let us also list the following useful properties of the
topologies $t_{\tau l}$ and $t_{w\tau l}$.

\medskip

\begin{prop}[\cite{Bik2}]\label{p_sv_top}
  Let $T_\alpha, \, S_\alpha \in S(\mathcal{M},\tau)$. Then
  \begin{itemize}
  \item [$(i)$] $T_\alpha \stackrel{t_{\tau l}}{\longrightarrow}$
    $\Longleftrightarrow$ $|T_\alpha| \stackrel{t_{\tau
        l}}{\longrightarrow}0$ $\Longleftrightarrow$ $|T_\alpha|^2
    \stackrel{t_{w\tau l}}{\longrightarrow}0$;
  \item[$(ii)$] if $T_\alpha \stackrel{t_{\tau
        l}}{\longrightarrow}T$, $S_\alpha \stackrel{t_{\tau
        l}}{\longrightarrow}S$, and the net $\{S_\alpha\}$ is
    $t_\tau$-bounded, then $T_\alpha S_\alpha\stackrel{t_{\tau
        l}}{\longrightarrow}T S$;
  \item[$(iii)$] if $0\leqslant S_\alpha\leqslant T_\alpha$ and
    $T_\alpha \stackrel{t_{w\tau l}}{\longrightarrow}0$, then
    $S_\alpha \stackrel{t_{w\tau l}}{\longrightarrow}0$.
  \end{itemize}
\end{prop}

\medskip

To compare the topologies $t_{w \tau l}$ and $t_{\tau l}$ with the
topology $t(\mathcal{M})$, we will need the following property of
the topology $t(\mathcal{M})$.

\begin{prop}\label{p_tl_ind}
  The topology $t(\mathcal{M})$ induced the topology
  $t(P\mathcal{M}P)$ on $LS(P\mathcal{M}P)$, where $0\neq P \in
  \mathcal{P}(\mathcal{M})$.
\end{prop}

\begin{proof}
  Let $\{Q_\alpha\}_{\alpha\in A}\subset \mathcal{P}(P\mathcal{M}P)$
  and $Q_\alpha \stackrel{t(\mathcal{M})}{\longrightarrow} 0$. By
  Proposition~\ref{plm-spk1}$(i)$ there exists a net
  $\{Z_\alpha\}_{\alpha\in A}\subset
  \mathcal{P}(\mathcal{Z}(\mathcal{M}))$ such that $Z_\alpha
  Q_\alpha \in \mathcal{P}_{fin}(\mathcal{M})$ for any $\alpha\in
  A$,
  $Z_\alpha^\bot\stackrel{t(\mathcal{Z}(\mathcal{M}))}{\longrightarrow}0$,
  and $\mathcal{D}(Z_\alpha
  Q_\alpha)\stackrel{t(\mathcal{Z}(\mathcal{M}))}{\longrightarrow}0$.
  The projection $R_\alpha=PZ_\alpha$ is in the center
  $\mathcal{Z}(P\mathcal{M}P)$ of the von Neumann algebra
  $P\mathcal{M}P$, and $R_\alpha Q_\alpha=Z_\alpha Q_\alpha$ is a
  finite projection in $P\mathcal{M}P$. Denote by $Z(P)$ the central
  support of the projection $P$. The mapping $\psi:
  P\mathcal{Z}(\mathcal{M})\rightarrow Z(P)\mathcal{Z}(\mathcal{M})$
  given by
  $$
  \psi(PZ)=Z(P)Z, \ \ Z\in \mathcal{Z}(\mathcal{M}),
$$
is a $*$-isomorphism from $P\mathcal{Z}(\mathcal{M})$ onto
$Z(P)\mathcal{Z}(\mathcal{M})$. Since
$\mathcal{Z}(P\mathcal{M}P)=P\mathcal{Z}(\mathcal{M})$~\cite[Sect.~3.1.5]{SZ.},
the $*$-algebras $\mathcal{Z}(P\mathcal{M}P)$ and
$Z(P)\mathcal{Z}(\mathcal{M})$ are \hbox{$*$-iso}morphic.

It is clear that
$$
\mathcal{D}_P(Q):\,=Z(P)\mathcal{D}(Q), \ \ \ Q\in
\mathcal{P}(P\mathcal{M}P),
$$
is a dimension function on $\mathcal{P}(P\mathcal{M}P)$, where
$\mathcal{D}$ is the initial dimension function on
$\mathcal{P}(\mathcal{M})$. We have that
$$
\mathcal{D}_P(R_\alpha Q_\alpha)=\mathcal{D}_P(Z_\alpha
Q_\alpha)=Z(P)\mathcal{D}(Z_\alpha
Q_\alpha)\stackrel{t(Z(P)\mathcal{Z}(\mathcal{M}))}{\longrightarrow}0.
$$
Moreover,
$$
P-R_\alpha=P(I-Z_\alpha)\stackrel{\psi}{\longmapsto}
Z(P)Z_\alpha^\bot\stackrel{t(Z(P)\mathcal{Z}(\mathcal{M}))}{\longrightarrow}0.
$$
Hence, by Proposition~\ref{plm-spk1} (i) we get that $Q_\alpha
\stackrel{t(P\mathcal{M}P)}{\longrightarrow}0$.

In the same way we can prove that $Q_\alpha
\stackrel{t(P\mathcal{M}P)}{\longrightarrow}0$, $\{Q_\alpha\}\subset
\mathcal{P}(P\mathcal{M}P)$, implies that $Q_\alpha
\stackrel{t(\mathcal{M})}{\longrightarrow}0$.

Let now $\{T_\alpha\}\subset LS(P\mathcal{M}P)$ and $T_\alpha
\stackrel{t(\mathcal{M})}{\longrightarrow}0$. By
Proposition~\ref{plm-spk1}(ii), we have that
$E^\bot_\lambda(|T_\alpha|)
\stackrel{t(\mathcal{M})}{\longrightarrow}0$ for any $\lambda>0$,
where $\{E_\lambda(|T_\alpha|)\}$ is a family of spectral
projections for $|T_\alpha|$. Denote by
$\{E^P_\lambda(|T_\alpha|)\}$ the family of spectral projections for
$|T_\alpha|$ in $LS(P\mathcal{M}P)$, $\lambda>0$. It is clear that
$E_\lambda(|T_\alpha|)=P^\bot+E^P_\lambda(|T_\alpha|)$ and
$E^\bot_\lambda(|T_\alpha|)=P-E^P_\lambda(|T_\alpha|)$ for all
$\lambda>0$.  It follows from above that $P-E^P_\lambda(|T_\alpha|)
\stackrel{t(P\mathcal{M}P)}{\longrightarrow}0$ for all $\lambda>0$.
Hence, by Proposition~\ref{plm-spk1}(ii), it follows that $T_\alpha
\stackrel{t(P\mathcal{M}P)}{\longrightarrow}0$.

One can similarly prove that the convergence $T_\alpha
\stackrel{t(P\mathcal{M}P)}{\longrightarrow}0$ implies the
convergence $T_\alpha \stackrel{t(\mathcal{M})}{\longrightarrow}0$.
\end{proof}

\begin{thm}\label{t_tl>ttl}
  $t_{\tau l}\leqslant t(\mathcal{M},\tau)$.
\end{thm}

\begin{proof}
  If $\{T_\alpha\}\subset S(\mathcal{M},\tau)$ and $T_\alpha
  \stackrel{t(\mathcal{M})}{\longrightarrow}0$, then $|T_\alpha|^2
  \stackrel{t(\mathcal{M})}{\longrightarrow}0$. Let $P \in
  \mathcal{P}(\mathcal{M})$ and $\tau(P)<\infty$. By
  Proposition~\ref{p_tl_ind}, we have that $P|T_\alpha|^2 P
  \stackrel{t(P\mathcal{M}P)}{\longrightarrow}0$. Since
  $\tau(P)<\infty$, it follows that
  $LS(P\mathcal{M}P)=S(P\mathcal{M}P,\tau)$ and the topology
  $t(P\mathcal{M}P)$ coincides with the measure topology $t_{\tau}$, that is, $P|T_\alpha|^2 P
  \stackrel{t_{\tau}}{\longrightarrow}0$. By
  Proposition~\ref{p_red_top}, we get that $|T_\alpha|^2
  \stackrel{t_{w\tau l}}{\longrightarrow}0$. Hence, it follows from
  Proposition~\ref{p_sv_top}(i) that $T_\alpha \stackrel{t_{\tau
      l}}{\longrightarrow}0$.
\end{proof}

\begin{note}\label{r_por_3}
  It follows from Theorem~\ref{t_tl>ttl} that the inequalities
  $$
  t_{w\tau l} \leqslant t_{\tau l}\leqslant t(\mathcal{M},\tau)
  \leqslant t_{\tau}
$$
always hold. If $\mathcal{M}=\mathcal{B}(\mathcal{H})\times
L_\infty[0,\infty)$, $\tau((T,f))=tr\, T+\int\limits_0^\infty f
\,d\,\mu$, where $T\in \mathcal{B}_+(\mathcal{H})$, $0\leqslant f\in
L_\infty[0,\infty)$, $\mu$ is the linear Lebesgue measure on
$[0,\infty)$, $\dim \mathcal{H}=\infty$, then
$S(\mathcal{M},\tau)=\mathcal{B}(\mathcal{H})\times
S(L_\infty[0,\infty),\mu)$ and, in this case, the following strict inequalities hold:
$$
t_{w\tau l} < t_{\tau l}< t(\mathcal{M},\tau) < t_{\tau}.
$$

\end{note}

\medskip

To find necessary and sufficient conditions for the topology
$t(\mathcal{M},\tau)$ to coincide with the topologies $t_{\tau l}$
and $t_{w\tau l}$, we will need the following.

\begin{prop}\label{p_t_ekv}
  Let $\mathcal{M}$ be a semifinite von Neumann algebra, $\tau$ be a
  faithful normal semifinite trace on $\mathcal{M}$,
  $\{T_\alpha\}_{\alpha\in A}\subset \mathcal{M}$,
  $\sup\limits_{\alpha\in A}\|T_\alpha\|_{\mathcal{M}}\leqslant 1$.
  \begin{itemize}
  \item [1)] If $\tau(I)<\infty$, then $T_\alpha \stackrel{t_{\tau
      }}{\longrightarrow}0$ if and only if
    $\tau(|T_\alpha|)\longrightarrow 0$.
  \item [2)] If the algebra $\mathcal{M}$ if finite and
    $\Phi: \mathcal{M}\mapsto \mathcal{Z}(\mathcal{M})$ is a
    center-valued trace on $\mathcal{M}$, then the following
    conditions are equivalent:
    \begin{itemize}
    \item [$(i)$] $T_\alpha \stackrel{t_{\tau
          l}}{\longrightarrow}0$;
    \item [$(ii)$] $\Phi(|T_\alpha|)
      \stackrel{t(\mathcal{Z}(\mathcal{M}))}{\longrightarrow}0$;
    \item [$(iii)$] $T_\alpha
      \stackrel{t(\mathcal{M})}{\longrightarrow}0$.
    \end{itemize}
  \end{itemize}
\end{prop}

\begin{proof}
  1). If $\tau(|T_\alpha|)\rightarrow 0$, then it follows at once
  from the inequality $\tau(\{|T_\alpha|>\lambda\})\leqslant
  \frac{1}{\lambda}\tau(|T_\alpha|)$, $\lambda> 0$, that $T_\alpha
  \stackrel{t_{\tau }}{\longrightarrow}0$ (here we do not need the
  condition that $\sup\limits_{\alpha\in
    A}\|T_\alpha\|_{\mathcal{M}}\leqslant 1$). Conversely, let
  $T_\alpha \stackrel{t_{\tau }}{\longrightarrow}0$. Then for every
  $\varepsilon>0$ there exist $\alpha(\varepsilon)\in A$ and
  $P_\alpha \in \mathcal{P}(\mathcal{M})$, $\alpha\geqslant
  \alpha(\varepsilon)$, such that
$$
\tau(P^\bot_\alpha)\leqslant \varepsilon, \ \ T_\alpha P_\alpha\in
\mathcal{M}, \ \ \|T_\alpha P_\alpha\|_{\mathcal{M}}\leqslant
\varepsilon.
$$
Consequently,
$\||T_\alpha|P_\alpha\|_{\mathcal{M}}\leqslant\varepsilon$ and
$\tau(|T_\alpha|P_\alpha)\leqslant\varepsilon\tau(I)$. Whence,
$$
\tau(|T_\alpha|) \leqslant
\varepsilon\tau(I)+\tau(|T_\alpha|P_\alpha^\bot) \leqslant
\varepsilon\tau(I)+\varepsilon\sup\limits_{\alpha\in
  A}\|T_\alpha\|_{\mathcal{M}}
$$
for all $\alpha\geqslant \alpha(\varepsilon)$, that is,
$\tau(|T_\alpha|)\rightarrow 0$.

2). $(i)\Rightarrow(ii)$. If $T_\alpha \stackrel{t_{\tau
    l}}{\longrightarrow}0$, then $|T_\alpha| \stackrel{t_{\tau l
  }}{\longrightarrow}0$ (Proposition~\ref{p_sv_top}) and thus
$|T_\alpha| \stackrel{t_{w\tau l }}{\longrightarrow}0$. Let
$\mathcal{P}_\tau(\mathcal{M})=\{P\in \mathcal{P}(\mathcal{M}):
\tau(P)<\infty\}$. For every finite subset $\beta=\{P_1, P_2,\ldots,
P_n\}\subset \mathcal{P}_\tau(\mathcal{M})$, let
$Q_\beta=\sup\limits_{1\leqslant i\leqslant n}P_i$. Denote by
$B=\{\beta\}$ the directed set of all finite subsets of
$\mathcal{P}_\tau(\mathcal{M})$, ordered by inclusion. It is clear
that $Q_\beta\uparrow I$ and $Q_\beta\in
\mathcal{P}_\tau(\mathcal{M})$ for all $\beta\in B$.

Let $V, \ U$ be neighborhoods in
$(S(\mathcal{Z}(\mathcal{M})),t(\mathcal{Z}(\mathcal{M})))$ of zero
such that $V+V\subset U$ and $XV\subset V$ for any $X\in
\mathcal{Z}(\mathcal{M})$ with $\|X\|_{\mathcal{M}}\leqslant 1$.
Since $\Phi(Q^\bot_\beta)\downarrow 0$, there exists $\beta_0\in B$
such that $\Phi(Q^\bot_{\beta_0})\in V$. Since
$$
0\leqslant \Phi(Q^\bot_{\beta_0}|T_\alpha|Q^\bot_{\beta_0})\leqslant
\sup\limits_{\alpha\in A}\|T_\alpha\|_{\mathcal{M}}
\Phi(Q^\bot_{\beta_0})\leqslant \Phi(Q^\bot_{\beta_0}),
$$
we have that $\Phi(Q^\bot_{\beta_0}|T_\alpha|Q^\bot_{\beta_0})\in V$
for all $\alpha\in A$. Identify the center
$\mathcal{Z}(\mathcal{M})$ with $L_\infty(\Omega,\Sigma,\mu)$ and,
for $E\in\Sigma$, $\mu(E)<\infty$, consider a faithful normal finite
trace $\nu_E$ on $\chi_E\mathcal{M}$ defined by
$$
\nu_E(X)=\int\limits_E \Phi(X)d\mu.
 $$

 Since $|T_\alpha| \stackrel{t_{w \tau l}}{\longrightarrow}0$, we
 have that $X_\alpha=Q_{\beta_0}|T_\alpha|Q_{\beta_0}\stackrel{t_{\tau
    }}{\longrightarrow}0$. Consequently, $\chi_E X_\alpha \stackrel{
   t_{\tau}}{\longrightarrow}0$ and, hence, $\chi_E X_\alpha \stackrel{
   t_{\nu_E}}{\longrightarrow}0$. Using item~1) we see that
$$
\int\limits_E \Phi(X_\alpha)d\mu=\nu(\chi_EX_\alpha)\longrightarrow
0.
 $$
 Consequently,  $\Phi( X_\alpha) \stackrel{
   t(\mathcal{Z}(\mathcal{M}))}{\longrightarrow}0$. Hence, there
 exists $\alpha(V)\in A$ such that $\Phi(X_\alpha)\in V$ for all
 $\alpha\geqslant\alpha(V)$. Using that $\Phi(XY)=\Phi(YX)$ for $X,Y
 \in \mathcal{M}$~\cite[Sect.~7.11]{SZ.} we get that
$$
\Phi(|T_\alpha|)=\Phi(Q_{\beta_0}|T_\alpha|Q_{\beta_0})+\Phi(Q^\bot_{\beta_0}|T_\alpha|Q^\bot_{\beta_0})\in
V+V\subset U
$$
for $\alpha \geqslant \alpha(V)$, which implies the convergence
$\Phi( |T_\alpha|) \stackrel{
  t(\mathcal{Z}(\mathcal{M}))}{\longrightarrow}0$.

$(ii)\Rightarrow(iii)$. If $\Phi( |T_\alpha|) \stackrel{
  t(\mathcal{Z}(\mathcal{M}))}{\longrightarrow}0$, then it follows
from
$\Phi(E^\bot_\lambda(|T_\alpha|))\leqslant\frac{1}{\lambda}\Phi(|T_\alpha|)$
that
$\Phi(E^\bot_\lambda(|T_\alpha|))\stackrel{t(\mathcal{Z}(\mathcal{M}))}{\longrightarrow}0$
for all $\lambda>0$.

Setting $Z_\alpha=I$ and using Proposition~\ref{plm-spk1}(i) we get
that $E^\bot_\lambda( |T_\alpha|) \stackrel{
  t(\mathcal{M})}{\longrightarrow}0$ for all $\lambda>0$ and, hence,
$T_\alpha \stackrel{ t(\mathcal{M})}{\longrightarrow}0$, see
Proposition~\ref{plm-spk1}(ii).

The implication $(iii)\Rightarrow (i)$ follows from Theorem
~\ref{t_tl>ttl}.
\end{proof}

\begin{thm}\label{t_wtl=tl}
  Let $\mathcal{M}$ be a semifinite von Neumann algebra, $\tau$ be a
  faithful normal semifinite trace on $\mathcal{M}$. The following
  condition are equivalent:
  \begin{itemize}
  \item [$(i)$] $t_{w\tau l}=t(\mathcal{M},\tau)$;
  \item [$(ii)$] $t_{\tau l}=t(\mathcal{M},\tau)$;
  \item [$(iii)$] $\mathcal{M}$ is finite.
  \end{itemize}
\end{thm}

\begin{proof}
  $(i)\Rightarrow (ii)$. If $t_{w\tau l}=t(\mathcal{M},\tau)$, then
  the operation of multiplication in $(S(\mathcal{M},\tau),t_{w\tau
    l})$ is jointly continuous. In this case, as was shown
  in~\cite[Theorem~4.1]{Bik2}, $t_{\tau l}=t_{w\tau l}$ and
  $\mathcal{M}$ is of finite type. The implication
  $(ii)\Rightarrow(iii)$ is proved similarly.

  $(iii) \Rightarrow (i)$. Let $\mathcal{M}$ be a finite von Neumann algebra. Then
  $t_{w\tau l}=t_{\tau l}$~\cite[Theorem~4.1]{Bik2}. Let
  $\{T_\alpha\}\subset S(\mathcal{M},\tau)$ and $T_\alpha \stackrel{
    t_{w\tau l}}{\longrightarrow}0$. It follows from the identity
  $t_{w\tau l}=t_{\tau l}$ and Proposition~\ref{p_sv_top} that
  $|T_\alpha| \stackrel{ t_{w\tau l}}{\longrightarrow}0$.

  If $\lambda\geqslant 1$, we have that $0\leqslant
  E^\bot_\lambda(|T_\alpha|)\leqslant |T_\alpha|$, hence
  $E^\bot_\lambda(|T_\alpha|) \stackrel{ t_{w\tau
      l}}{\longrightarrow}0$ by Proposition~\ref{p_sv_top}. Using
  Proposition~\ref{p_t_ekv}, item~2, we get that
  $E^\bot_\lambda(|T_\alpha|) \stackrel{
    t(\mathcal{M})}{\longrightarrow}0$ for all $\lambda\geqslant 1$. Since
$$
E^\bot_\lambda(|T_\alpha|E^\bot_1(|T_\alpha|))=
\begin{cases}{ E^\bot_1(|T_\alpha|), \ \ \ \ \ \ \ 0<\lambda < 1,} \\
  {E^\bot_\lambda(|T_\alpha|), \ \ \ \ \ \ \ \lambda\geqslant 1,}
\end{cases}
$$
by Proposition~\ref{plm-spk1}$(ii)$ we get that
$|T_\alpha|E^\bot_1(|T_\alpha|) \stackrel{
  t(\mathcal{M})}{\longrightarrow}0$. Now, it follows from the
inequality $|T_\alpha|E_1(|T_\alpha|)\leqslant|T_\alpha|$ that
$|T_\alpha|E_1(|T_\alpha|) \stackrel{ t_{w\tau
    l}}{\longrightarrow}0$. Since $t_{w\tau l}=t_{\tau l}$ and
$\|\,|T_\alpha|E_1(|T_\alpha|)\,\|_{\mathcal{M}}\leqslant 1$, we
have that $|T_\alpha|E_1(|T_\alpha|) \stackrel{
  t(\mathcal{M})}{\longrightarrow}0$ by Proposition~\ref{p_t_ekv},
item~2.

Hence, $|T_\alpha|=|T_\alpha|E_1(|T_\alpha|)+|T_\alpha|
E^\bot_1(|T_\alpha|) \stackrel{ t(\mathcal{M})}{\longrightarrow}0$,
and so $T_\alpha \stackrel{ t(\mathcal{M})}{\longrightarrow}0$. With
a use of Theorem~\ref{t_tl>ttl} this shows that $t_{w\tau l}=
t(\mathcal{M},\tau)$.
\end{proof}

\section{Comparison of the topologies $t_{\tau l}$
  and $t_{w \tau
    l}$ with the $(o)$-topology
  on $S_h(\mathcal{M},\tau)$}

Let us denote by $t_{h \tau l}$ (resp., $t_{h w\tau l}$) the
topology on $S_h(\mathcal{M},\tau)$ induced by the topology $t_{\tau
  l}$ (resp., $t_{w \tau l}$), and find a
connection between these topologies and the $(o)$-topology
$t_{o\tau}(\mathcal{M})$.

\begin{prop}\label{p_t0>wtl}
  $t_{h w\tau l} \leqslant t_{h\tau l}\leqslant t_{o\tau}(\mathcal{M})$.
\end{prop}

\begin{proof}
  Let $\{T_\alpha\}_{\alpha\in A}\subset S_h(\mathcal{M},\tau)$,
  $T_\alpha\downarrow 0$, $P\in \mathcal{P}(\mathcal{M})$,
  $\tau(P)<\infty$. Since $(PT_\alpha P)\downarrow 0$, we have that
  $PT_\alpha P \stackrel{t_{h \tau}}{\longrightarrow}0$ by
  Proposition~\ref{p_(o)=ttau}. Consequently, $T_\alpha
  \stackrel{t_{w \tau l}}{\longrightarrow}0$ by
  Proposition~\ref{p_red_top}. Let $0\leqslant S_\alpha\leqslant
  T_\alpha$, $S_\alpha \in S_h(\mathcal{M},\tau)$. By
  Proposition~\ref{p_sv_top}(iii), $S_\alpha \stackrel{t_{w \tau
      l}}{\longrightarrow}0$ and, hence, $\sqrt{S_\alpha}
  \stackrel{t_{\tau l}}{\longrightarrow}0$ by
  Proposition~\ref{p_sv_top}(i). Let us show that $S_\alpha
  \stackrel{t_{ \tau l}}{\longrightarrow}0$.

  Let $\mu_t(T)=\inf\{\|TP\|_{\mathcal{M}}: P\in
  \mathcal{P}(\mathcal{M}), \ \tau(P^\bot)\leqslant t\}$, $t>0$, be
  a non-increasing rearrangement of the operator $T$. Fix
  $\alpha_0\in A$. For every $\alpha\geqslant\alpha_0$, we have
  $$
  \mu_t(\sqrt{S_\alpha})=\sqrt{\mu_t(S_\alpha)}\leqslant
  \sqrt{\mu_t(T_\alpha)}\leqslant \sqrt{\mu_t(T_{\alpha_0})},
$$
in particular,
$$
\sup\limits_{\alpha\geqslant\alpha_0}\mu_t(\sqrt{S_\alpha})\leqslant
\sqrt{\mu_t(T_{\alpha_0})}<\infty
$$
for all $t>0$. Consequently, the net
$\{\sqrt{S_\alpha}\}_{\alpha\geqslant\alpha_0}$ is
$t_\tau$-bounded~\cite[Lemma~1.2]{Bik2} and, hence, $S_\alpha
\stackrel{t_{\tau l}}{\longrightarrow}0$ by
Proposition~\ref{p_sv_top}(ii). Repeating now the
proof of the implication $(ii)\Rightarrow (i)$ in Theorem
~\ref{t_(o)=tl} we get that $t_{h\tau l}\leqslant
t_{o\tau}(\mathcal{M})$.

The inequality $t_{h w \tau l}\leqslant t_{h \tau l}$ follows from
the inequality $t_{ w \tau l}\leqslant t_{\tau l}$.
\end{proof}

\begin{corr}\label{c_twtaul=ttau 1}
  \begin{itemize}
  \item [$(i)$] If $t_{w\tau l}=t_\tau$ (resp., $t_{\tau
      l}=t_\tau$), then $\tau(I)<\infty$.
  \item [$(ii)$] If $t_{h w\tau l}=t_{h\tau l}$, then the algebra
    $\mathcal{M}$ is finite.
  \item [$(iii)$] If $t_{h w \tau l}=t_{h \tau}$ (resp. $t_{h\tau
      l}=t_{h \tau}$), then $\tau(I)<\infty$.
  \end{itemize}
\end{corr}

\begin{proof}
  $(i)$. It follows from $t_{w\tau l}=t_\tau$ that $t_{\tau
    l}=t_\tau$. Consequently, $t_{h \tau}\leqslant
  t_{o \tau}(\mathcal{M})$ and, hence, $\tau(I)<\infty$ by
  Proposition~\ref{p_(o)=ttau}.

  $(ii)$. If $T_\alpha \stackrel{t_{w \tau l}}{\longrightarrow}0$,
  then $T_\alpha^* \stackrel{t_{w \tau l}}{\longrightarrow}0$ and,
  hence,
  $$
  Re \,T_\alpha=\frac{1}{2}(T_\alpha + T^*_\alpha) \stackrel{t_{w
      \tau l}}{\longrightarrow}0, \ \ Im
  \,T_\alpha=\frac{1}{2i}(T_\alpha - T^*_\alpha) \stackrel{t_{w \tau
      l}}{\longrightarrow}0.
 $$
 Since $t_{h w\tau l}=t_{h\tau l}$, we get that $T_\alpha
 \stackrel{t_{ \tau l}}{\longrightarrow}0$. Hence, $t_{w \tau l}=t_{
   \tau l}$, which implies that $\mathcal{M}$ is finite ~\cite{Bik2}.

 Item $(iii)$ follows from Propositions~\ref{p_(o)=ttau}
 and~\ref{p_t0>wtl}.
\end{proof}

\begin{corr}\label{c_twtaul=ttau}
  The following conditions are equivalent:
  \begin{itemize}
  \item [$(i)$] $t_{w\tau l}=t_\tau$;
  \item [$(ii)$] $t_{\tau l}=t_\tau$;
  \item [$(iii)$] $\tau(I)< \infty$.
  \end{itemize}
\end{corr}

\begin{proof}
  The implication $(i)\Rightarrow (ii)$ follows from the
  inequalities $t_{w\tau l}\leqslant t_{\tau l}\leqslant t_\tau$,
  $(ii)\Rightarrow (iii)$ from Propositions~\ref{p_(o)=ttau}
  and~\ref{p_t0>wtl}, and the implication $(iii)\Rightarrow (i)$ is clear.
\end{proof}

\medskip

By Proposition~\ref{p_(o)=ttau}, if $\tau(I)<\infty$, we have the
following:
$$
t_{h w \tau l}=t_{h\tau l}=t_{o\tau}(\mathcal{M})=t_{h\tau}.
$$

\medskip

The following theorem permits to construct examples of von Neumann
algebras $\mathcal{M}$ for which
$$
t_{h w\tau l}< t_{h\tau l}< t_{o \tau}(\mathcal{M})<t_{h\tau}.
$$

\begin{thm}\label{t_th=tot}
  If $t_{h\tau l}=t_{o\tau}(\mathcal{M})$, then $\tau(I)<\infty$.
\end{thm}

\begin{proof}
  Assume that $\tau(I)=+\infty$ and first consider the
  $\sigma$-finite von Neumann algebra $\mathcal{M}$. In this case,
  there is a faithful normal positive linear functional $\varphi$ on
  $\mathcal{M}$~\cite[Ch.~II, $\S\,3$]{Take1.}. If we have a net
  $\{T_\alpha\}\subset \mathcal{M}_+$ and $T_\alpha\downarrow 0$,
  then $\varphi(T_\alpha)\downarrow 0$ and, hence, there is a
  sequence of indices $\alpha_1\leqslant\alpha_2\leqslant\ldots$
  such that $\varphi(T_{\alpha_n})\downarrow 0$, which implies
  $T_{\alpha_n}\downarrow 0$.

  Let a net $\{T_\alpha\}\subset S_h(\mathcal{M},\tau)$
  $(o)$-converge to zero in $S_h(\mathcal{M},\tau)$. Choose
  $S_\alpha\in S_+(\mathcal{M},\tau)$ such that $-S_\alpha\leqslant
  T_\alpha\leqslant S_\alpha$ and $S_\alpha\downarrow 0$. Fix
  $\alpha_0$ and set $X=(I+S_{\alpha_0})^{-\frac{1}{2}}$,
  $Y_\alpha=XS_\alpha X$, $\alpha\geqslant \alpha_0$. Then $Y_\alpha
  \in \mathcal{M}_+$, $Y_\alpha\downarrow 0$, and hence there is a
  sequence $\alpha_1\leqslant\alpha_2\leqslant\ldots$ such that
  $Y_{\alpha_n}\downarrow 0$. Consequently, $S_{\alpha_n}\downarrow
  0$ and $T_{\alpha_n} \stackrel{ (o)}{\longrightarrow}0$.

  Hence, the subset $F\subset S_h(\mathcal{M},\tau)$ is closed in
  the $(o)$-topology $t_{o\tau}(\mathcal{M})$ if and only if $F$
  contains $(o)$-limits of all $(o)$-convergent sequences of
  elements in $F$.

  Choose a sequence $\{P_n\}$ of nonzero pairwise orthogonal
  projections in $\mathcal{P}(\mathcal{M})$ satisfying $1\leqslant
  \tau(P_n)<\infty$ and show that $F=\{\sqrt{n}P_n\}_{n=1}^\infty$
  is closed in the $(o)$-topology $t_{o\tau}(\mathcal{M})$.

  If $\{T_k\}_{k=1}^\infty\subset F$ is an $(o)$-convergent sequence
  of pairwise distinct elements, then
  $T_k=\sqrt{n_k}P_{n_k}\leqslant S$, $k=1,2,\ldots$, for some $S\in
  S_+(\mathcal{M},\tau)$ and, hence, $0\leqslant P_{n_k}\leqslant
  \frac{1}{\sqrt{n_k}}S\stackrel{ t_\tau}{\longrightarrow}0$.
  Consequently, $\tau(P_{n_k})\rightarrow 0$, which contradicts the
  inequality $\tau(P_{n_k})\geqslant 1$, $k=1,2,\ldots$

  Hence, the set $F$ is closed in the $(o)$-topology
  $t_{o\tau}(\mathcal{M})$.

  It remains to show that this set $F=\{\sqrt{n}P_n\}_{n=1}^\infty$
  is not closed in the topology $t_{w\tau l}$ (here we do not use
  that the algebra $\mathcal{M}$ is $\sigma$-finite).

  Denote by $\mathcal{M}_*^+$ the set of all positive normal linear
  functionals on $\mathcal{M}$, and let $t_\sigma$ be the
  $\sigma$-strong topology on $\mathcal{M}$ generated by the family of
  seminorms $p_\psi(T)=\psi(T^*T)^{\frac{1}{2}}$, $\psi\in
  \mathcal{M}_*^+$, $T\in \mathcal{M}$~\cite[Ch.~II,
  $\S\,2$]{Take1.}.  It is clear that the linear functional
  $\varphi_Q(T)=\tau(QTQ)$ belongs to $\mathcal{M}_*^+$ for all
  $Q\in \mathcal{P}_{\tau}(\mathcal{M})$ such that $\tau(Q)<\infty$. Thus,
  the convergence $T_\alpha\stackrel{ t_\sigma}{\longrightarrow}0$
  implies that $\tau(QT_\alpha^*T_\alpha
  Q)=\varphi_Q(T_\alpha^*T_\alpha)=p^2_{\varphi_Q}(T_\alpha)\rightarrow 0$, that
  is, $|T_\alpha Q|^2=QT^*_\alpha T_\alpha Q \stackrel{
    t_\tau}{\longrightarrow}0$. By~\cite{Tish} we have that
  $|T_\alpha Q|\stackrel{ t_\tau}{\longrightarrow}0$ and, hence,
  $T_\alpha Q\stackrel{ t_\tau}{\longrightarrow}0$, that is,
  $T_\alpha\stackrel{ t_{\tau l}}{\longrightarrow}0$ by
  Proposition~\ref{p_red_top}. Consequently, the topology $t_\sigma$
  majorizes the topology $t_{\tau l}$ on $\mathcal{M}$.

  Let us now show that $T=0$ belongs to the closure of the set $F$
  in the topology $t_\sigma$. The sets
$$
V(\varphi_1,\ldots,\varphi_n,\varepsilon)=\{T\in \mathcal{M}: p_{\varphi_i}(T)\leqslant \varepsilon, \ i=1,2,\ldots, n\}
$$
form a neighborhood base around zero in the topology $t_\sigma$, where
$\{\varphi_i\}_{i=1}^n\subset \mathcal{M}^+_*$, $\varepsilon>0$,
$n\in \mathbf{N}$. If $\varphi=\sum\limits_{i=1}^n \varphi_i$, then
$\varphi\in \mathcal{M}^+_*$ and $p_{\varphi_i}(T)\leqslant
p_\varphi(T)$, $i=1,2,\ldots, n$. Hence, the system of subsets
$\{V(\varphi,\varepsilon): \varphi\in \mathcal{M}^+_*,
\varepsilon>0\}$ is a neighborhood base around zero in the topology
$t_\sigma$. If $V(\varphi,\varepsilon)\cap F=\emptyset$, then
$\varepsilon<p_\varphi(\sqrt{n}P_n)=\sqrt{n}\varphi(P_n)^{\frac{1}{2}}$
and, hence, $\varphi(P_n)>\frac{\varepsilon^2}{n}$ for all
$n=1,2,\ldots$, which is impossible, since
$\sum\limits_{n=1}^\infty\varphi(P_n)=\varphi(\sup\limits_{n\geqslant
  1}P_n)<+\infty$.

Consequently, $V(\varphi,\varepsilon)\cap F\neq \emptyset$ for all
$\varphi\in \mathcal{M}^+_*$, $\varepsilon>0$. This means that $T=0$
belongs to the closure of the set $F$ in the topology $t_\sigma$.
Since $t_\sigma$ majorizes the topology $t_{\tau l}$ on
$\mathcal{M}$, zero belongs to the closure of the set $F$ in the
topology $t_{\tau l}$. Consequently, the set $F$ is not closed in
$(S_h(\mathcal{M},\tau),t_{h\tau l})$ and, hence, $t_{h\tau l}<
t_{o\tau}(\mathcal{M})$.

Let now $\mathcal{M}$ be a not $\sigma$-finite von Neumann algebra,
$\{P_n\}_{n=1}^\infty\subset \mathcal{P}(\mathcal{M})$ be as before,
$P=\sup\limits_{n\geqslant 1}P_n$, and $\mathcal{A}=P\mathcal{M}P$.
It is clear that $\varphi(T)=\sum\limits_{n=1}^\infty\frac{\tau(P_n
  T P_n)}{2^n\tau(P_n)}$, $T\in \mathcal{A}$, is a faithful normal
linear functional on $\mathcal{A}$, and thus the algebra
$\mathcal{A}$ is $\sigma$-finite~\cite[Ch.~II, $\S\,3$]{Take1.}.

Let $\{T_\alpha\}_{\alpha\in A}\subset S_h(\mathcal{A},\tau)$, $T\in
S_h(\mathcal{M},\tau)$, and
$T_\alpha\stackrel{(o)}{\longrightarrow}T$ in
$S_h(\mathcal{M},\tau)$, that is, there is a net
$\{S_\alpha\}_{\alpha\in A}\subset S_+(\mathcal{M},\tau)$ such that
$-S_\alpha \leqslant T_\alpha-T \leqslant S_\alpha$ and
$S_\alpha\downarrow 0$. Then
$$
-PS_\alpha P\leqslant P(T_\alpha-T)P=T_\alpha -PTP\leqslant
PS_\alpha P
$$
and $PS_\alpha P\downarrow 0$, that is,
$T_\alpha\stackrel{(o)}{\longrightarrow}PTP$ in
$S_h(\mathcal{A},\tau)$ and in $S_h(\mathcal{M}, \tau)$.
Consequently, $T=PTP$ so that $T\in S_h(\mathcal{A},\tau)$. This
means that $S_h(\mathcal{A},\tau)$ is closed in
$(S_h(\mathcal{M},\tau), t_{o\tau}(\mathcal{M}))$, and the
$(o)$-topology $t_{o\tau}(\mathcal{M})$ induces the $(o)$-topology
$t_{o\tau}(\mathcal{A})$ on $S_h(\mathcal{A},\tau)$. In particular,
the set $F=\{\sqrt{n}P_n\}_{n=1}^\infty$ is closed in
$(S_h(\mathcal{M},\tau), t_{o\tau}(\mathcal{M}))$, although it is
not closed in the topology $t_{h\tau l}$.
\end{proof}

\medskip

Proposition~\ref{p_t0>wtl} and Theorem~\ref{t_th=tot} immediately
give the following.

\begin{corr}\label{c_str_ner}
  \begin{itemize}
  \item [$(i)$] If $t_{hw\tau l}=t_{o\tau}(\mathcal{M})$, then $\tau(I)<\infty$.

  \item [$(ii)$] If  $\mathcal{M}$ is not finite, then
    $ t_{h w\tau l}< t_{h\tau l}< t_{o \tau}(\mathcal{M})<t_{h\tau}.$
\item[$(iii)$] If $\mathcal{M}$ is finite and
  $\tau(I)=+\infty$, then
  $t_{h w\tau l}= t_{h\tau l}< t_{o \tau}(\mathcal{M})<t_{h\tau}.$
  \end{itemize}
\end{corr}

\section{ The locally measure topology on atomic algebras}

Necessary and sufficient conditions on the algebra $\mathcal{M}$ so
that the topology $t_{\tau l}$ would be locally convex (resp.,
normable) were given in the paper of A.~M.~Bikchentaev~\cite{Bik3}.
Let us give a similar criterion for the topology $t(\mathcal{M})$.

A nonzero projection $P\in \mathcal{P}(\mathcal{M})$ is called an
atom if $0\neq Q \leqslant P$, $Q\in \mathcal{P}(\mathcal{M})$,
implies that $Q=P$.

A von Neumann algebra $\mathcal{M}$ is atomic if every nonzero
projection in $\mathcal{M}$ majorizes some atom. Any atomic von
Neumann algebra $\mathcal{M}$ is $*$-isomorphic to the
$C^*$-product
$$
C^*-\prod\limits_{j\in J}\mathcal{M}_j=\{\{T_j\}_{j\in J}: T_j\in
\mathcal{M}_j, \ \sup\limits_{j\in
  J}\|T_j\|_{\mathcal{M}_j}<+\infty\},
$$
where $\mathcal{M}_j=\mathcal{B}(\mathcal{H}_j)$, $j\in J$. Since
$LS(\mathcal{B}(\mathcal{H}_j))=\mathcal{B}(\mathcal{H}_j)$ and
$LS(C^*-\prod\limits_{j\in J}\mathcal{M}_j)=\prod\limits_{j\in
  J}LS(\mathcal{M}_j)$~\cite[Ch.~II, $\S3$]{Mur_m}, we have that, for an atomic von
Neumann algebra $\mathcal{M}$, the $*$-algebra $LS(\mathcal{M})$ is
$*$-isomorphic to the direct product $\prod\limits_{j\in
  J}\mathcal{B}(\mathcal{H}_j)$ of the algebras
$\mathcal{B}(\mathcal{H}_j)$. By Proposition~\ref{plm-spk1}, the
topology $t(\mathcal{M})$ coincides with the Tychonoff product of the
topologies $t(\mathcal{B}(\mathcal{H}_j))$. Since
$t(\mathcal{B}(\mathcal{H}_j))$ is a uniform topology
$t_{\|\cdot\|_{\mathcal{B}(\mathcal{H}_j)}}$ on
$\mathcal{B}(\mathcal{H}_j)$ generated by the norm
$\|\cdot\|_{\mathcal{B}(\mathcal{H}_j)}$, the topology
$t(\mathcal{M})$ is locally convex. For every $0\leqslant
\{T_j\}_{j\in J}\in C^*-\prod\limits_{j\in
  J}\mathcal{B}(\mathcal{H}_j)$, set $\tau(\{T_j\}_{j\in
  J})=\sum\limits_{j\in J} tr_j(T_j)$, where $tr_j$ is the canonical
trace on $\mathcal{B}(\mathcal{H}_j)$. It is clear that $\tau$ is a
faithful normal semifinite trace on the atomic von Neumann algebra
$\mathcal{M}=C^*-\prod\limits_{j\in J}\mathcal{M}_j$, and the
topology $t_{\tau l}$ is also locally convex~\cite{Bik3}, however,
$t_{\tau l}\neq t(\mathcal{M})$ if $ \dim \mathcal{H}_j=\infty$ for
at least one index $j\in J$.

\begin{prop}\label{p_tm_lv}
  The topology $t(\mathcal{M})$ is locally convex if and only if
  $\mathcal{M}$ is $*$-isomorphic to the $C^*$-product
  $C^*-\prod\limits_{j\in J}\mathcal{M}_j$, where $\mathcal{M}_j$
  are factors of type $I$ or type $III$.
\end{prop}

\begin{proof}
  Let $t(\mathcal{M})$ be a locally convex topology on
  $LS(\mathcal{M})$. Since $t(\mathcal{M})$ induces the topology
  $t(\mathcal{Z}(\mathcal{M}))$ on $\mathcal{Z}(\mathcal{M})$, we
  have that $(S(\mathcal{Z}(\mathcal{M})),
  t(\mathcal{Z}(\mathcal{M})))$ is a locally convex space. It
  follows from~\cite[Ch.~V, $\S3$]{Sar.} that
  $\mathcal{Z}(\mathcal{M})$ is an atomic von Neumann algebra.
  Hence, the algebra $\mathcal{M}$ is $*$-isomorphic to the
  $C^*$-product $C^*-\prod\limits_{j\in J}\mathcal{M}_j$, where
  $\mathcal{M}_j$ are factors for all $j\in J$. Let $M_{j_0}$ be of
  type $II$-factor. Then there exists a nonzero finite projection
  $P\in \mathcal{P}(\mathcal{M})$ such that $P\mathcal{M}P$ is of
  type $II_1$. It follows
  from ~\cite[Ch.~V~$\S\,3$]{Sar.}  that $S(P\mathcal{M}P, t(P\mathcal{M}P))$ has not nonzero continuous linear
  functional and,
  hence, the topology $t(P\mathcal{M}P)$ can not be locally convex.
  By Proposition~\ref{p_tl_ind}, the topology $t(\mathcal{M})$ can
  not be locally convex too. Consequently, $\mathcal{M}_j$ are
  either of type $I$ or type $III$ factors for all $j\in J$.

  Conversely, let $\mathcal{M}=C^*-\prod\limits_{j\in
    J}\mathcal{M}_j$, where $\mathcal{M}_j$ are of type $I$ or type
  $III$ factors. Then $LS(\mathcal{M}_j)=\mathcal{M}_j$,
  $t(\mathcal{M}_j)=t_{\|\cdot\|_{\mathcal{M}_j}}$,
  $LS(\mathcal{M})=\prod\limits_{j\in J}\mathcal{M}_j$ and, hence,
  the topology $t(\mathcal{M})$ is a Tychonoff product of the normed
  topologies $t(\mathcal{M}_j)$, that is, $t(\mathcal{M})$ is a
  locally convex topology.
\end{proof}

\begin{corr}\label{c_t_norm}
  The topology $t(\mathcal{M})$ can be normed if and only if
  $\mathcal{M}=\prod\limits_{j=1}^n \mathcal{M}_j$, where
  $\mathcal{M}_j$ are of type $I$ or type $III$ factors, $j=1,2,...,n
  $, and $n$ is a positive integer.
\end{corr}

\begin{proof}
  If the topology $t(\mathcal{M})$ is normable, then
  $(S(\mathcal{Z}(\mathcal{M})), t(\mathcal{Z}(\mathcal{M})))$ is a
  normable vector space. It follows from~\cite[Ch.~V,~$\S\,3$]{Sar.}
  that $\mathcal{Z}(\mathcal{M})$ is a finite dimensional algebra,
  which implies that $\mathcal{M}=\prod\limits_{j=1}^n
  \mathcal{M}_j$, where $\mathcal{M}_j$ are factors, $j=1,2,...,n$.
  By Proposition~\ref{p_tm_lv}, the factors $\mathcal{M}_j$ are
  either of type $I$ or type $III$ for all $j=1,2,..., n$.

  The converse implication is obvious.
\end{proof}

\medskip

Let us also mention one more useful property of the topologies
$t(\mathcal{M})$ if $\mathcal{M}$ is an atomic finite algebra.

\begin{prop}\label{p_tm_conat}
  The following conditions are equivalent:
  \begin{itemize}
  \item [$(i)$] $\mathcal{M}$ is an atomic finite von Neumann algebra;
  \item [$(ii)$] if $\{T_n\}_{n=1}^\infty\subset LS_h(\mathcal{M})$,
    then $T_n \stackrel{t(\mathcal{M})}{\longrightarrow}0$ if and
    only if $T_n \stackrel{(o)}{\longrightarrow}0$.
  \end{itemize}
\end{prop}

\begin{proof}
  $(i)\Rightarrow(ii)$. Since $\mathcal{M}$ is a finite von
  Neumann algebra, it follows from $T_n
  \stackrel{(o)}{\longrightarrow}0$ that $T_n
  \stackrel{t(\mathcal{M})}{\longrightarrow}0$ by
  Theorem~\ref{t_(o)=tl}. Since $\mathcal{M}$ is atomic, we have
  $\mathcal{M}=C^*-\prod\limits_{j\in J}\mathcal{B}(\mathcal{H}_j)$.
  If $T_n=\{T_n^{(j)}\}_{j\in J}$, $T_n^{(j)}\in
  \mathcal{B}(\mathcal{}H_j)$, and $T_n
  \stackrel{t(\mathcal{M})}{\longrightarrow}0$, then
  $\|T_n^{(j)}\|_{\mathcal{B}(\mathcal{H}_j)}\rightarrow 0$ as $n\to
  \infty$ for all $j\in J$. Since $|T_n^{(j)}|\leqslant
  \|T_n^{(j)}\|_{\mathcal{B}(\mathcal{H}_j)}\cdot
  I_{\mathcal{B}(\mathcal{H}_j)}$, it follows that
  $\{T_n^{(j)}\}_{n=1}^\infty$ $(o)$-converges to zero in
  $\mathcal{B}(\mathcal{H}_j)$ and, hence, $T_n
  \stackrel{(o)}{\longrightarrow}0$.

  $(ii)\Rightarrow(i)$. It follows from Remark~\ref{t_(o)=tl} that
  $\mathcal{M}$ is finite. Identify the center
  $\mathcal{Z}(\mathcal{M})$ with $L_\infty(\Omega,\Sigma,\mu)$. By
  condition~$(ii)$, any sequence in $L_0(\Omega,\Sigma,\mu)$ that
  $\mu$-almost everywhere converges is convergent in the topology $t(\mathcal{M})$.
   Consequently, $\mathcal{Z}(\mathcal{M})$ is an atomic von
  Neumann algebra and, hence, $\mathcal{M}=C^*-\prod\limits_{j\in J}
  \mathcal{M}_j$, where $\mathcal{M}_j$ are finite factors of types
  $I$ or $II$. If there is an index $j_0\in J$ for which
  $\mathcal{M}_{j_0}$ is of type $II$, then there exists a nonzero
  projection $P\in \mathcal{P}(\mathcal{M})$, $\tau(P)=1$, such that
  $P\mathcal{M}P$ is of type $II_1$.

  Let $\mathcal{A}$ be a maximal commutative $*$-subalgebra of
  $P\mathcal{M}P$. Then $\mathcal{A}$ has no atoms and there is a
  collection $Q_n^{(k)}$, $1\leqslant k \leqslant n$, of pairwise
  orthogonal projections in $\mathcal{P}(\mathcal{A})$ such that
  $\sup\limits_{1\leqslant k\leqslant n}Q_n^{(k)}=P$ and
  $\tau(Q_n^{(k)})=\frac{1}{n}, \ k=1,2,\ldots, n$. Set $X^{(k)}_n=n
  Q_n^{(k)}$, $k=1,2,\ldots,n$, and index the operators $X_n^{(k)}$
  by setting $T_1=X_1^{(1)}, \ T_2=X_2^{(1)}, \ T_3=X_2^{(2)},
  \ldots$ It is clear that $T_n \stackrel{t_\tau}{\longrightarrow}0$
  and, by condition~$(ii)$, $T_n \stackrel{(o)}{\longrightarrow}0$
  in $LS_h(P\mathcal{M}P)$, that is, there exists a sequence
  $\{S_n\}_{n=1}^\infty \subset LS_+(P\mathcal{M}P)$ such that
  $S_n\downarrow 0 $ and $0\leqslant T_n\leqslant S_n$,
  $n=1,2,\ldots$

  Since $E_n^\bot(S_1)\in P\mathcal{M}P$ and
  $E_n^\bot(S_1)\downarrow 0$, there is an index $n_0$ such that
  $\tau(E_{n_0}^\bot(S_1))<\frac{1}{2}$. Set $E=PE_{n_0}(S_1)$ and
  $L_n=ES_n E$. It is clear that $\frac{1}{2}<\tau(E)\leqslant 1$,
  $0\leqslant L_n \leqslant n_0 E$, $L_n\downarrow 0$, and
  $0\leqslant E T_n E \leqslant L_n$, $n=1,2,\ldots$

  Since $\tau(L_n)\downarrow 0$, the inequality
  $\tau(E^\bot_\varepsilon(L_n))\leqslant
  \frac{1}{\varepsilon}\tau(L_n)$ implies that
  $\tau(E^\bot_\varepsilon(L_n))\rightarrow 0$ for all
  $\varepsilon>0$.

  Fix $\varepsilon\in (0,1)$ and choose an index $n_1$ such that
  $\tau(E^\bot_\varepsilon(L_{n_1}))<\frac{1}{2}$. For the
  projection $G=PE_\varepsilon(L_{n_1})$, we have that $GET_n
  EG\leqslant GL_n G\leqslant GL_{n_1}G\leqslant \varepsilon G$ for
  all $n\geqslant n_1$, that is,
  $\|Q_n^{(k)}EG\|_{\mathcal{M}}\leqslant
  \sqrt{\frac{\varepsilon}{n}}$ for any $n\geqslant n_1$,
  $k=1,2,\ldots, n$. If $E\wedge G= 0$, then $1=\tau(P)\geqslant
  \tau(E \vee G)=\tau(E)+\tau(G)>1$. Consequently, $E\wedge G\neq
  0$, so that there exists a vector $\xi\in (E\wedge
  G)(\mathcal{H})\subset P(\mathcal{H})$ with
  $\|\xi\|_{\mathcal{H}}=1$, where $\|\cdot\|_{\mathcal{H}}$ is the
  norm on the Hilbert space $\mathcal{H}$ on which the von Neumann
  algebra $\mathcal{M}$ acts. For each $n\geqslant n_1$, we have
  that
  $$
  1=\|P\xi\|^2_{\mathcal{H}}=\sum
  \limits_{k=1}^n\|Q_n^{(k)}\xi\|^2_{\mathcal{H}}=
  \sum\limits_{k=1}^n\|(Q_n^{(k)}EG)\xi\|^2_{\mathcal{H}}\leqslant
  \frac{\varepsilon}{n}+\frac{\varepsilon}{n}+
  \ldots+\frac{\varepsilon}{n}=\varepsilon<1.
  $$
  This contradiction shows that the sequence $\{T_n\}_{n=1}^\infty$
  can not be \hbox{$(o)$-con}vergent to zero in
  $LS_h(P\mathcal{M}P)$.  Hence, all the factors $\mathcal{M}_j$,
  $j\in J$, are of type $I$, that is, $\mathcal{M}$ is an atomic von
  Neumann algebra.
\end{proof}

\begin{corr}\label{c_M-atom}
  The following conditions are equivalent.
  \begin{itemize}
  \item [$(i)$]  Any $t_{w\tau l}$-convergent sequence in
    $S_h(\mathcal{M},\tau)$ is $(o)$-convergent.
  \item [$(ii)$] Any $t_{\tau l}$-convergent sequence in
    $S_h(\mathcal{M},\tau)$ is $(o)$-convergent.
  \item [$(iii)$] Any $t_{\tau }$-convergent sequence in
    $S_h(\mathcal{M},\tau)$ is $(o)$-convergent.
  \item [$(iv)$] $\mathcal{M}$ is an atomic von Neumann algebra and
    $\tau(I)<\infty$.
  \end{itemize}
\end{corr}

\begin{proof}
  The implications $(i)\Rightarrow (ii)\Rightarrow (iii)$ follow
  from the inequalities $t_{w\tau l}\leqslant t_{\tau l}\leqslant
  t_{\tau }$.

  $(iii)\Rightarrow (iv)$. Since the topology $t_\tau$ is
  metrizable, it follows from Remark~\ref{r_por_1} that
  $t_{h\tau}=t_{o\tau}$ and, hence, $\tau(I)<\infty$ by
  Proposition~\ref{p_(o)=ttau}, in particular,
  $t_\tau=t(\mathcal{M})$. It remains to apply
  Proposition~\ref{p_tm_conat}.

  The implication $(iv)\Rightarrow (i)$ follows from
  Theorem~\ref{t_wtl=tl} and Proposition~\ref{p_tm_conat}.
\end{proof}

\begin{note}\label{r_neosh}
  It was shown in the proof of the implication $(ii)\Rightarrow(i)$
  in Proposition~\ref{p_tm_conat} that for a non-atomic von Neumann
  algebra $\mathcal{M}$ with a faithful normal trace $\tau$ there
  always exists a sequence $\{E_n\}_{n=1}^\infty$ of pairwise
  commuting projections in $\mathcal{P}(\mathcal{M})$ such that $E_n
  \stackrel{t_\tau}{\longrightarrow}0$, however,
  $\{E_n\}_{n=1}^\infty$ does not $(o)$-converge in
  $S_h(\mathcal{M},\tau)$.
\end{note}

\bigskip
\noindent\begin{tabular}{ll} Vladimir Chilin &  \\
National University of Uzbekistan, &
\\Tashkent, 100174, Republica of Uzbekistan. \hspace{0.2cm} &
\\E-mail: chilin@ucd.uz &
\end{tabular}

\bigskip
\noindent\begin{tabular}{ll} Mustafa Muratov &  \\
Taurida National University, &
\\Simferopol, 95007, Ukraine.  \hspace{0.2cm} &
\\E-mail: mustafa-muratov@mail.ru &
\end{tabular}

\end{document}